\documentclass[12pt]{amsart}
\usepackage{amssymb,latexsym}
\usepackage{enumerate}

\makeatletter
\@namedef{subjclassname@2010}{%
  \textup{2010} Mathematics Subject Classification}
\makeatother

\newtheorem{theorem}{Theorem}[section]
\newtheorem{corollary}[theorem]{Corollary}
\newtheorem{lemma}[theorem]{Lemma}
\newtheorem{proposition}[theorem]{Proposition}
\theoremstyle{definition}

\newtheorem{remark}[theorem]{Remark}
\newtheorem{example}[theorem]{Example}

\numberwithin{equation}{section}

\usepackage{amssymb}
\usepackage[utf8]{inputenc}
\usepackage[T1]{fontenc}
\usepackage{courier}
\newcommand{\fff}{\mathcal{F}}
\newcommand{\sss}{\mathcal{S}}
\newcommand{\aaa}{\mathcal{A}}
\newcommand{\bbb}{\mathcal{B}}
\newcommand{\eee}{\mathcal{E}}
\newcommand{\nn}{\mathbb{N}}
\newcommand{\fin}{[\mathbb{N}]^{<\omega}}
\newcommand{\infin}{[\mathbb{N}]}
\newcommand{\supp}{\text{supp}}
\newcommand{\ran}{\text{ran}}
\newtheorem{claim}{Claim}

\begin{document}

\baselineskip=17pt

\title[Concerning the Szlenk index]{Concerning the Szlenk index}

\author[RM Causey]{Ryan M Causey}
\address{Department of Mathematics \\ University of South Carolina \\
Columbia, SC 29208}
\email{CAUSEYRM@mailbox.sc.edu}

\date{}

\begin{abstract}
We discuss pruning and coloring lemmas on regular families.  We discuss several applications of these lemmas to computing the Szlenk index of certain $w^*$ compact subsets of the dual of a separable Banach space.  Applications include estimates of the Szlenk index of Minkowski sums, infinite direct sums of separable Banach spaces, constant reduction, and three space properties.  

We also consider using regular families to construct Banach spaces with prescribed Szlenk index.  As a consequence, we give a characterization of which countable ordinals occur as the Szlenk index of a Banach space, prove the optimality of a previous universality result, and compute the Szlenk index of the injective tensor product of separable Banach spaces.  
 
\end{abstract}



\maketitle

\tableofcontents

\addtocontents{toc}{\setcounter{tocdepth}{1}}

\section{Introduction}
 
A classical result in Banach space theory is that every separable Banach space embeds isometrically in $\mathcal{C}[0,1]$.  One can ask whether other classes of Banach spaces, for example the class of Banach spaces having separable dual, admit a member which contains isomorphic copies of every member of that class.  For the case of Banach spaces having separable dual, Szlenk \cite{Sz} introduced the Szlenk index to prove that there is no Banach space having separable dual which contains isomorphic copies of all Banach spaces having separable dual.  Since its inception, the Szlenk index has been the object of significant investigation.  

Typically defined in terms of slicings of the unit ball of the dual of a separable Banach space, the Szlenk index of a separable Banach space is equal to the weakly null $\ell_1^+$ index of that space in the case that this space does not contain a copy of $\ell_1$ \cite{AJO}.  This fact allows for a modification of  certain transfinite versions of an argument of James \cite{James} involving equivalence of finite representability and crude finite representability of $\ell_1$ in a Banach space.  This argument can be used to yield new information about the Szlenk index and new methods for estimating it.  More generally, regular families play a key role in computing so-called $\sigma$-indices in separable Banach spaces.  Consequently, certain purely combinatorial results concerning colorings of regular families have as easy corollaries strong results about Szlenk index, including that of \cite{AJO}.  Moreover, regular families can be used to construct Banach spaces with prescribed weakly null $\ell_1^+$ behavior, which can be used to prove certain existence and non-existence results.  For example, we provide a characterization of which countable ordinals occur as the Szlenk index of a Banach space.  In \cite{Causey1}, it was shown that for each countable ordinal $\xi$ there exists a separable Banach space with Szlenk index $\omega^{\xi+1}$ which contains isomorphic copies of every separable Banach space having Szlenk index not exceeding $\omega^\xi$.  By being able to construct a Banach space with precise control over the weakly null $\ell_1^+$ index, we are able to prove the optimality of that result.  

In the first half of the paper, we discuss regular families, colorings and prunings thereof, and applications of these coloring results to computing the Szlenk index of certain subsets of the dual of a separable Banach space.  We generalize Alspach, Judd, and Odell's argument that the Szlenk index of a Banach space not containing $\ell_1$ is equal to its weakly null $\ell_1^+$ index in order to compute the Szlenk index of certain sets $K\subset X^*$, $X$ a separable Banach space.  We then deduce as easy applications of this work a number of corollaries, some old and some new.  In the second half of the paper, we discuss how to construct Banach spaces with prescribed weakly null $\ell_1^+$ structure.  As a consequence, we provide a characterization of the countable ordinals which occur as the Szlenk index of a Banach space and use this to prove the optimality of the universality results of \cite{Causey1} and \cite{Causey2}.  We also show how one can compute the Szlenk index of a Banach space having separable dual via embeddings into Banach spaces with shrinking basis having subsequential upper block estimates in certain mixed Tsirelson spaces.  With this, we prove an optimal result about the Szlenk index of an injective tensor product of two separable Banach spaces.  

The paper is arranged as follows.  In Section $2$, we discuss the necessary definitions concerning Banach spaces and finite dimensional decompositions.  In Section $3$, we discuss trees, regular families, and their use in computing ordinal indices. In this section we also discuss two useful pruning lemmas which will be used throughout.  In Section $4$, we state and prove the combinatorial lemmas concerning regular families.  In Section $5$, we define the Szlenk and $\ell_1^+$ weakly null indices and provide several examples of applications thereof.  In Section $6$, we discuss the use of mixed Tsirelson spaces in constructing Banach spaces with prescribed $\ell_1^+$ behavior and the special role played by these families.

\section{Banach spaces and finite dimensional decompositions}

If $X$ is a Banach space, we say a sequence $E=(E_n)$ of finite dimensional subspaces of $X$ is a \emph{finite dimensional decomposition} (FDD) for $X$ provided that for each $x\in X$, there exists a unique sequence $(x_n)$ so that $x_n\in E_n$ for each $n\in \nn$ and $x=\sum x_n$.  In this case, for each $n\in \nn$, the operator $x=\sum x_m\mapsto x_n$ is a bounded linear operator from $X$ to $E_n$, called the $n^{th}$ \emph{canonical projection}, denoted $P^E_n$.  For a finite set $A$, we let $P_A=\sum_{n\in A}P_n$.  By the principle of uniform boundedness, the \emph{projection constant of} $E$ \emph{in} $X$, given by $\sup_{m\leqslant n} \|P^E_{[m,n]}\|$, is finite.  We say $E$ is \emph{bimonotone} for $X$ if the projection constant of $E$ in $X$ is $1$.  It is well-known that if $E$ is an FDD for $X$, one can equivalently renorm $X$ to make $E$ a bimonotone FDD for $X$ with the new norm.   Throughout, we will assume that for each $n\in \nn$, $E_n\neq \{0\}$.  

We can consider $E_n^*$ as being embedded in $X^*$ via the adjoint $(P^E_n)^*$, although this embedding is not necessarily isometric unless $E$ is bimonotone.  We let $E^*=(E_n^*)$, and consider these as subspaces of $X^*$.  The FDD $E$ is said to be \emph{shrinking} for $X$ if $E^*$ is an FDD for $X^*$.  Since $E^*$ will always be an FDD for the closed span $[E_n^*]_{n\in \nn}$ with projection constant in this space not exceeding the projection constant of $E$ in $X$, $E$ is a shrinking FDD for $X$ if and only if $X^*=[E_n^*]_{n\in \nn}$.     

If $E$ is an FDD for $X$ and if $0=s_0<s_1<\ldots$, and $F_n=[E_k]_{s_{n-1}<k\leqslant s_n}$, then $F=(F_n)$ is called a \emph{blocking of} $E$.  In this case, $F$ is also an FDD for $X$ with projection constant in $X$ not exceeding the projection constant of $E$ in $X$.  If $E$ is shrinking, any blocking of $E$ will be as well.    

If $x\in X$, we let $\supp_E(x)=\{n\in \nn: P_n^Ex\neq 0\}$.  We let $\ran_E(x)$ be the smallest interval in $\nn$ which contains $\supp_E(x)$.  We let $c_{00}(E)=\{x\in X: |\supp_E(x)|<\infty\}$.  We say a (finite or infinite) sequence of non-zero vectors $(x_n)$ is a \emph{block sequence with respect to } $E$ provided $\max \supp_E(x_n)< \min \supp_E(x_{n+1})$ for each appropriate $n$.  

We let $\Sigma(E, X)$ denote all finite block sequences with respect to $E$ in $B_X$.  We say $\mathcal{B}\subset \Sigma(E,X)$ is a \emph{hereditary block tree} in $X$ with respect to $E$ if it contains all subsequences of its members.  If $\overline{\varepsilon}=(\varepsilon_i)\subset (0,1)$ and if $\mathcal{B}$ is a hereditary block tree, we let $$\mathcal{B}_{\overline{\varepsilon}}^{E,X}=\{(x_i)_{i=1}^n\in \Sigma(E,X): n\in \nn, \exists (y_i)_{i=1}^n\in \mathcal{B}, \|x_i-y_i\|< \varepsilon_i \text{\ }\forall 1\leqslant i\leqslant n\}.$$  If $(\varepsilon_i)$ is non-increasing, $\mathcal{B}_{\overline{\varepsilon}}^{E,X}$ is also a hereditary block tree in $X$ with respect to $E$.  

Given (finite or infinite) sequences $(e_n), (f_n)$ of the same length in (possibly different) Banach spaces, we say $(e_n)$ $C$-dominates $(f_n)$, or that $(f_n)$ is $C$-dominated by $(e_n)$, provided that for each $(a_n)\in c_{00}$, $$\Bigl\|\sum a_nf_n\Bigr\|\leqslant C \Bigl\|\sum a_ne_n\Bigr\|.$$  

If $E$ is an FDD for a Banach space $X$ and if $(e_n)$ is a normalized, $1$-unconditional basis for the Banach space $U$, we say $E$ \emph{satisfies subsequential} $C$-$U$ \emph{upper block estimates in} $X$ provided that for any normalized block sequence $(x_n)$ with respect to $E$, if $m_n=\max \supp_E(x_n)$, $(x_n)$ is $C$-dominated by $(e_{m_n})$.  This idea has occurred in other works, such as \cite{OSZ}, \cite{FOSZ}, and \cite{Causey1}, where $m_n$ was taken to be $\min \supp_E(x_n)$ rather than the maximum. Our definition is chosen for convenience within this work, and it does not affect the main theorems contained herein, or the main theorems contained in the cited works.  This is because for each basis $(e_n)$ considered in the main theorems of the cited works, and for each pair of sequences of natural numbers $k_1<k_2<\ldots$, $l_1<l_2<\ldots$ so that $\max\{k_n, l_n\}< \min \{k_{n+1}, l_{n+1}\}$, $(e_{k_n})$ and $(e_{l_n})$ are equivalent.

\begin{proposition} Let $X$ be a Banach space not containing $\ell_1$.  \begin{enumerate}[(i)] \item Suppose $Y\leqslant X$ is a closed subspace, $(x_n)\subset B_X $ is weakly null, and $\delta\in (0,1/2)$ is such that $\|x_n\|_{X/Y}<\delta$ for all $N\in \nn$.  Then there exists a weakly null sequence $(y_n)\subset B_Y$ and a subsequence $(x_{k_n})$ of $(x_n)$ so that for each $n\in \nn$, $\|x_{k_n}-y_n\|<4\delta$.  \item If $Q:X\to Z$ is a quotient map and $(z_n)\subset B_Z$ is weakly null, then for any $\delta>0$, there exists a weakly null sequence $(x_n)\subset 3B_X$ and a subsequence $(z_{k_n})$ of $(z_n)$ so that for all $n\in \nn$, $\|z_{k_n}-Qx_n\|< \delta$.  \end{enumerate}

\label{perturb}

\end{proposition}

\begin{proof} Several times, we will use Rosenthal's $\ell_1$ dichotomy \cite{Rosenthal}, which states that any bounded sequence in a Banach space either has a subsequence equivalent to the canonical $\ell_1$ basis or a subsequence which is weakly Cauchy.

(i) For each $n$, choose some $u_n\in Y$ so that $\|x_n-u_n\|<\delta$. By passing to a subsequence, we can assume that $(u_n)$ is weakly Cauchy.  Choose a convex block defined by $v_n=\sum_{i\in I_n}a_ix_i$ so that $\|v_n\|< \delta-\|x_n-u_n\|$.  Let $w_n=\sum_{i\in I_n}a_iu_i$.  Then $(u_n-w_n)$ is weakly null in $Y$ and $$\|u_n-w_n\|\leqslant \|x_n\|+\|x_n-u_n\|+ \|v_n\| + \sum_{i\in I_n}a_i\|x_i-u_i\| \leqslant 1+2\delta.$$  Moreover, \begin{align*} \|x_n-(u_n-w_n)\| & \leqslant \|x_n-v_n-(u_n-w_n)\|+ \|v_n\| \\ & \leqslant \|x_n-u_n\|+ \|v_n\|+ \sum_{i\in I_n}a_i\|x_i-u_i\|<2\delta.\end{align*} Then if $y_n$ is the normalization of $u_n-w_n$, $$\|x_n-y_n\| \leqslant \|x_n-(u_n-w_n)\| + \|y_n-(u_n-w_n)\| < 4\delta.$$  Since $\|u_n-w_n\|\geqslant 1-2\delta$, $(y_n)$ is seminormalized, and therefore also weakly null.  

(ii) Choose $\varepsilon>0$ to be determined.  For each $n\in \nn$, choose $u_n\in X$ with $\|u_n\|< 1+\varepsilon$ so that $Qu_n=z_n$.  By passing to a subsequence, we can assume $(u_n)$ is weakly Cauchy.  Choose a convex block $v_n=\sum_{i\in I_n}a_iz_i$ so that $\|v_n\|<\varepsilon$.  Let $w_n=\sum_{i\in I_n}a_iu_i$.  Then $$\|u_n-w_n\|\leqslant 1+\varepsilon + \sum_{i\in I_n}a_i(1+\varepsilon)=2+2\varepsilon<3$$ for appropriate $\varepsilon$.  Moreover, this sequence is weakly null.  Last, $$\|z_n-Q(u_n-w_n)\|= \|Qw_n\| =\|v_n\|< \varepsilon< \delta.$$  Thus taking $\varepsilon < \min \{1/2, \delta\}$ suffices.

\end{proof}
 
\section{Trees, derivatives, and indices}

\subsection{Trees on sets}

Throughout, if $P, Q$ are partially ordered sets, we say $f:P\to Q$ is \emph{order preserving} provided that if $x,y\in P$ with $x<_Py$, $f(x)<_Q f(y)$.  We say $f:P\to Q$ is an \emph{embedding} if it is a bijection so that for $x,y\in P$, $x<_Py$ if and only if $f(x)<_Qf(y)$.  

Given a set $S$, we let $S^\omega$ (resp. $S^{<\omega}$) denote the set of all infinite (resp. finite) sequences in $S$. We include the sequence of length zero, denoted $\varnothing$, in $S^{<\omega}$.  For $s\in S^{<\omega}$, we let $|s|$ denote the length of $s$. For $s,t\in S^{<\omega}$, we let $s\verb!^!t$ denote the concatenation of $s$ with $t$.    Given $s=(x_i)_{i=1}^n\in S^{<\omega}$, we let $s|_m=(x_i)_{i=1}^m$ for $0\leqslant m\leqslant n$.  We define the partial order $\prec$ on $S^{<\omega}$ by $s\prec s'$ provided $|s|< |s'|$ and $s=s'|_{|s|}$. If $s\prec s'$, we say $s$ is a \emph{predecessor} of $s'$, and $s'$ is a \emph{successor} of $s$.  If $|s'|=|s|+1$, we say $s$ is the \emph{immediate predecessor} of $s'$, and $s'$ is an \emph{immediate successor} of $s$.  Given a set $U\subset S^{<\omega}$, we let $C(U)$ denote the set of all finite, non-empty chains in $U\setminus \{\varnothing\}$.   We define a partial order $<$ on $C(U)$ by $c<c'$ provided $s\prec s'$ for each $s\in c$, $s'\in c'$.  

If $T\subset S^{<\omega}$ is downward closed with respect to the order $\prec$, we call $T$ a \emph{tree}, and we let $MAX(T)$ denote the maximal elements of $T$ with respect to the order $\prec$. We let $\widehat{T}=T\setminus \{\varnothing\}$. If $T$ contains all subsequences of its members, we say $T$ is \emph{hereditary}. If $T\subset S^{<\omega}$, we let $$T(s)=\{t\in S^{<\omega}: s\verb!^!t\in T\},$$ and note that if $T$ is a tree (resp. hereditary tree), $T(t)$ is a tree (resp. hereditary tree) as well. If $T$ is a tree, we call linearly ordered subsets of $T$ \emph{segments} of $T$, and maximal segments will be called \emph{branches} of $T$.  If $T$ is a tree on a vector space, we say $T$ is \emph{convex} provided it contains all convex blockings of its members.  We recall that for a sequence $(x_i)_{i=1}^n$ in a vector space, $(y_i)_{i=1}^m$ is a convex blocking of $(x_i)_{i=1}^n$ provided there exist $0=k_0<\ldots <k_m= n$ and non-negative scalars $(a_i)_{i=1}^n$ so that for each $j$, $\sum_{i=k_{j-1}+1}^{k_j} a_i=1$ and $y_j=\sum_{i=k_{j-1}+1}^{k_j} a_ix_i$.  

Given a tree $T$, we let $d(T)=T\setminus MAX(T)$, and note that this is a tree as well.  We define the countable transfinite derivations as follows. Throughout, $\omega, \omega_1$ will denote the first infinite and uncountable ordinals, respectively.  We let $$d^0(T)=T,$$ $$d^{\xi+1}(T)=d(d^\xi(T)),\text{\ \ } \xi<\omega_1,$$ and 
$$d^\xi(T)=\underset{\zeta<\xi}{\bigcap} d^\zeta(T), \xi<\omega_1 \text{\ a limit ordinal.}$$  Finally, we define the order $o$ of the tree $T$ by $$o(T)=\min \{\xi<\omega_1: d^\xi(T)=\varnothing\}$$ provided such a $\xi$ exists, and $o(T)=\omega_1$ otherwise.  

\subsection{Regular trees on $\nn$}

Throughout, if $M$ is any infinite subset of $\nn$, we let $[M]$ (resp. $[M]^{<\omega}$) denote the infinite (resp. finite) subsets of $M$.  We identify the subsets of $\nn$ in the natural way with strictly increasing sequences in $\nn$.  We topologize the power set of $\nn$ by identifying it with the Cantor set.  A set $\fff\subset \fin$ is called compact if it is compact with respect to this topology.  For $E,F\subset \nn$, we write $E<F$ to denote $\max E<\min F$.  For $n\in \nn$ and $E\subset \nn$, we write $n\leqslant E$ to denote $n\leqslant \min E$.  By convention, we let $\varnothing <E <\varnothing$ for any $E$.  Throughout, we will write $E\verb!^! F$ in place of $E\cup F$ in the case that $E<F$.  We write $n\verb!^!E$ (resp. $E\verb!^!n$) in place of $(n)\verb!^!E$ (resp. $E\verb!^! (n)$).    

Given $(k_i)_{i=1}^n, (l_i)_{i=1}^n\in \fin$, we say $(l_i)_{i=1}^n$ is a \emph{spread} of $(k_i)_{i=1}^n$ provided $k_i\leqslant l_i$ for each $1\leqslant i\leqslant n$.  We say $\fff\subset \fin$ is \emph{spreading} provided it contains all spreads of its members. With the identification of sets with sequences, we can naturally identify such a family with a tree on $\nn$, and we say $\fff$ is hereditary if it hereditary as a tree. We call a family $\fff\subset \fin$ \emph{regular} provided it is compact, spreading, and hereditary.

We say that a sequence $(E_i)_{i=1}^n\subset \fin$ is $\fff$ \emph{admissible} if it is successive and $(\min E_i)_{i=1}^n\in \fff$.  Given a regular family $\mathcal{G}$ and a set $E$, we say the successive sequence $(E_i)_{i=1}^n$ is the \emph{standard decomposition of} $E$ \emph{with respect to} $\mathcal{G}$ provided that $E=\cup_{i=1}^n E_i$ and for each $j\leqslant n$, $E_j$ is the maximal initial segment of $\cup_{i=j}^n E_i$ which is a member of $\mathcal{G}$.  Note that $E$ admits a standard decomposition with respect to $\mathcal{G}$ if and only if $(\min E)\in \mathcal{G}$, and in this case the standard decomposition is unique.   

If $(m_n)=M\in \infin$, the bijection $n\mapsto m_n$ induces a natural bijection between the power sets of $\nn$ and $M$, which we also denote $M$.  That is, $M(E)=(m_n: n\in E)$.  For $\fff\subset \fin$, we let $\mathcal{F}(M)=\{M(E):E\in \fff\}$.  If $M\in \infin$ and if $\fff\subset \fin$, we let $M^{-1}(\fff)=\{E: M(E)\in \fff\}$.  

Given regular families $\fff, \mathcal{G}$, we define $$(\fff, \mathcal{G})=\{F\verb!^!G: F\in\fff, G\in \mathcal{G}\},$$ \begin{align*} \fff[\mathcal{G}] & =\Bigl\{\underset{i=1}{\overset{n}{\bigcup}}E_i: E_1<\ldots<E_n, E_i\in \mathcal{G}, (\min E_i)_{i=1}^n\in \fff\Bigr\}\\ & = \Bigl\{\underset{i=1}{\overset{n}{\bigcup}}E_i: (E_i)_{i=1}^n\subset \mathcal{G} \text{\ is\ }\fff \text{\ admissible}\Bigr\}.\end{align*} We observe that a set $E\in\fff[\mathcal{G}]$ if and only if $E$ has an $\fff$ admissible standard decomposition $(E_i)_{i=1}^n$ with respect to $\mathcal{G}$. For a given $\fff$, we let $[\fff]^1=\fff$ and $[\fff]^{n+1}=\fff\bigl[[\fff]^n\bigr]$ for $n\in\nn$.  

 If $(\mathcal{G}_n)$ is a sequence of regular families, we let $$\mathcal{D}(\mathcal{G}_n)=\{E: \exists n\leqslant E\in \mathcal{G}_n\}.$$ We think of $(\fff, \mathcal{G})$ as the sum of the trees $\fff, \mathcal{G}$, $\fff[\mathcal{G}]$ as the product of $\fff, \mathcal{G}$, and $\mathcal{D}(\mathcal{G}_n)$ as the diagonalization of the families $\mathcal{G}_n$.  

For each $1\leqslant n$, let $\mathcal{A}_n=\{E\in \fin: |E|\leqslant n\}$ and $\mathcal{S}=\mathcal{D}(\mathcal{A}_n)$.  If $\zeta\leqslant \omega_1$ is a limit ordinal, we say that the family $(\mathcal{G}_\xi)_{0\leqslant \xi<\zeta}$ is \emph{additive} if for each $\xi<\zeta$, $\mathcal{G}_{\xi+1}=(\mathcal{A}_1, \mathcal{G}_\xi)$ and for each limit ordinal $\xi<\zeta$, there exists $\xi_n\uparrow \xi$ so that $\mathcal{G}_\xi=\mathcal{D}(\mathcal{G}_{\xi_n})$.  We say $(\mathcal{G}_\xi)_{0\leqslant \xi<\zeta}$ is \emph{multiplicative} if for each $\xi<\zeta$, $\mathcal{G}_{\xi+1}=\mathcal{S}[\mathcal{G}_\xi]$, and $(1)\in MAX(\mathcal{G}_0)$.  Observe in this case that $(1)\in MAX(\mathcal{G}_\xi)$ for every $\xi<\zeta$.  

If $\fff$ is regular, we observe that $\fff'$ is also regular, and $MAX(\fff)$ is the set of isolated points in $\fff$.  Thus $\fff'$ is the Cantor-Bendixson derivative of $\fff$.    In place of the Cantor-Bendixson index, we define the index $$\iota(\fff)=\min \{\xi<\omega_1: \fff^\xi\subset \{\varnothing\}\}.$$  It is easy to see that for $\fff$ hereditary, this set or ordinals is non-empty if and only if $\fff$ is compact, which is equivalent to $\fff$ not containing any infinite chain.  Moreover, if $\fff\neq \varnothing$, $\iota(\fff)+1$ coincides with the Cantor-Bendixson derivative of $\fff$.  The justification for using the index $\iota$ in place of the Cantor-Bendixson index is evident in the following proposition.  

\begin{proposition}Let $\fff, \mathcal{G}$, and $\mathcal{G}_n$ be non-empty regular families.  \begin{enumerate}[(i)]  \item For $0\leqslant \zeta,\xi<\omega_1$, $(\fff^\zeta)^\xi= \fff^{\zeta+\xi}$.  \item $(\fff, \mathcal{G})$ is regular and $\iota(\fff, \mathcal{G})=\iota(\mathcal{G})+\iota(\fff)$. \item  $\fff[\mathcal{G}]$ is regular and $\iota(\fff[\mathcal{G}])=\iota(\mathcal{G})\iota(\fff)$. \item For any $M\in \infin$, $M^{-1}(\fff)$ is regular and $\iota(M^{-1}(\fff))=\iota(\fff)$.\item For any $M\in \infin$, $M^{-1}(\fff[\mathcal{G}])= M^{-1}(\fff)[M^{-1}(\mathcal{G})]$. \item $\mathcal{D}(\mathcal{G}_n)$ is regular and $\iota(\mathcal{D}(\mathcal{G}_n))=\sup_n \iota(\mathcal{G}_n)$ if this supremum is not attained. \item If $M\in \infin$ and $\iota(\fff)\leqslant \iota(\mathcal{G})$, there exists $N\in [M]$ so that $\fff(N)\subset \mathcal{G}$. \item If $\zeta<\omega_1$ is a limit ordinal and $(\mathcal{G}_\xi)_{0\leqslant \xi<\zeta}$ is either additive or multiplicative, then for each $0\leqslant \xi\leqslant \eta<\zeta$, there exist $m,n\in \nn$ so that $\mathcal{G}_\xi\cap [m, \infty)^{<\omega}\subset \mathcal{G}_\eta$ and $\mathcal{G}_\xi\subset \mathcal{G}_{\eta+n}$.  \end{enumerate}

\label{regular facts}\end{proposition}

\begin{proof}   

(i) By induction on $\xi$ for $\zeta$ fixed.  The $\xi=0$ and successor cases are trivial.  If $\xi$ is a limit ordinal, $\zeta+\xi$ is also a limit, so $$(\fff^\zeta)^\xi = \underset{\eta<\xi}{\bigcap} (\fff^\zeta)^\eta = \underset{\eta<\xi}{\bigcap} \fff^{\zeta+\eta} = \underset{\eta<\zeta+\xi}{\bigcap} \fff^\eta = \fff^{\zeta+\xi}.$$  Here we have used that $\eta\mapsto \zeta+\eta$ is continuous and that the Cantor-Bendixson derivatives of $\fff$ are decreasing.

(ii) It is clear that a subset (resp. spread) of $F\verb!^!G$, $F\in \fff, G\in \mathcal{G}$, can be written in the form $F_0\verb!^! G_0$ where $F_0$ (resp. $G_0$) is a subset (resp. spread) of $F$ (resp. $G$).  Thus $(\fff, \mathcal{G})$ is spreading and hereditary.  If $N|_n\in (\fff, \mathcal{G})$ for all $n\in \nn$, let $m\in \nn$ be maximal so that $N_m\in \fff$.  Then choose $n\in \nn$ maximal so that $(N\setminus N|_m)|_n\in \mathcal{G}$.  It is clear that for any $k>n+m$, $N|_k\notin (\fff, \mathcal{G})$.  This is because if $F\verb!^! G=N|_k$, then either $F$ is a proper extension of $N|_m$ or $G$ has a subset which is a proper extension of $(N\setminus N|_m)|_n$, either of which contradict the maximality of either $m$ or $ n$. 

Next, we note that for $F\in MAX(\fff)$, $(\fff, \mathcal{G})(F)= \mathcal{G}\cap (\max F, \infty)^{<\omega}$.  Since $\iota(\mathcal{G}\cap (\max F, \infty))=\iota(\mathcal{G})$, $(\varnothing) = (\fff, \mathcal{G})(F)^{\iota(\mathcal{G})}$, which means $F\in MAX((\fff, \mathcal{G})^{\iota(\mathcal{G})})$.  If $E\in (\fff, \mathcal{G})\setminus \fff$, $E=F\verb!^!G$ for $F\in MAX(\fff)$ and $\varnothing \neq G\in \mathcal{G}$.  The above argument shows that $E\notin (\fff, \mathcal{G})^{\iota(\mathcal{G})}$.  Therefore $\fff=(\fff, \mathcal{G})^{\iota(\mathcal{G})}$, and $\iota((\fff, \mathcal{G}))=\iota(\mathcal{G})+\iota(\fff)$.  

(iii) Any spread (resp. subset) of $\cup_{i=1}^n E_i$ is an $\fff$ admissible union of spreads (resp. subsets) $F_i$ of $E_i$.  If $N|_n\in \fff[\mathcal{G}]$ for all $n\in \nn$, choose recursively $n_0, n_1, n_2, \ldots$ maximal so that $n_0=0$ and $(N\setminus N|_{n_{i-1}})|_{n_i}\in \mathcal{G}$ for all $i\in \nn$.  Let $m_i=\min (N\setminus N|_{n_{i-1}})$ and choose $k$ so that $(m_i)_{i=1}^k\notin \fff$.  Then for any $s>\sum_{i=1}^k n_i$, $N|_s\notin \fff[\mathcal{G}]$.  Indeed, if $N|_s\in \fff[\mathcal{G}]$, let $(E_i)_{i=1}^t$ be the standard decomposition of $N|_s$ with respect to $\mathcal{G}$.  Then $\fff\ni (\min E_i)_{i=1}^t$ is a proper extension of $(m_i)_{i=1}^k$, a contradiction.  

We prove by induction that $\fff[\mathcal{G}]^{\iota(\mathcal{G})\xi} = \fff^\xi[\mathcal{G}]$. The result is clear if $\fff=\{\varnothing\}$ or $\mathcal{G}=\{\varnothing\}$, so assume $\iota(\fff), \iota(\mathcal{G})>0$. The base case is true by definition.  If $(E_i)_{i=1}^n\subset \mathcal{G}$ is $\fff$ admissible with $F:=(\min E_i)_{i=1}^n\in \fff'$, then there exists $m> \max E_n$ so that for each $i\geqslant m$, $F\verb!^!i\in \fff$.  Then $\mathcal{G}\cap (m, \infty)^{<\omega}\subset \fff[\mathcal{G}]\bigl(\cup_{i=1}^n E_i\bigr)$.  This means $\cup_{i=1}^n E_i\in \fff[\mathcal{G}]^{\iota(\mathcal{G})}$, whence $\fff'[\mathcal{G}]\subset \fff[\mathcal{G}]^{\iota(\mathcal{G})}$.  Next, fix $E\in \fff[\mathcal{G}]$ and let $(E_i)_{i=1}^n$ be the standard decomposition of $E$ with respect to $\mathcal{G}$. Suppose that $(\min E_i)_{i=1}^n\in MAX(\fff)$.   Then $\fff[\mathcal{G}]\bigl(\cup_{i=1}^n E_i\bigr) = \mathcal{G}(E_n)$.  But $\iota(\mathcal{G}(E_n))<\iota(\mathcal{G})$, which means $\cup_{i=1}^n E_i\notin \fff[\mathcal{G}]^{\iota(\mathcal{G})}$.  This means $\fff[\mathcal{G}]^{\iota(\mathcal{G})}\subset \fff'[\mathcal{G}]$, and these sets are equal.  Applying this argument again to $\fff^\xi$ in place of $\fff$ yields the successor case.  Last, for a limit ordinal $\xi$, $\iota(\mathcal{G})\xi$ is also a limit ordinal.  Then $$\fff[\mathcal{G}]^{\iota(\mathcal{G})\xi} = \underset{\zeta<\iota(\mathcal{G})\xi}{\bigcap} \fff[\mathcal{G}]^\zeta =\underset{\zeta<\xi}{\bigcap} \fff[\mathcal{G}]^{\iota(\mathcal{G})\zeta} = \underset{\zeta<\xi}{\bigcap} \fff^\zeta[\mathcal{G}]= \fff^\xi[\mathcal{G}].$$  The last equality follows from the fact that $E$ will lie in either of the two sets if and only if $E$ has a maximal decomposition $(E_i)_{i=1}^n$ with respect to $\mathcal{G}$ and that this sequence is $\fff^\xi$ admissible, while this second property is equivalent to being $\fff^\eta$ admissible for every $\zeta<\xi$.  

(iv) If $E\in M^{-1}(\fff)$ and $F$ is a subset (resp spread) of $E$, $M(F)$ is a subset (resp. spread) of $M(E)$. Therefore $M(F)\in \fff$, whence $F\in M^{-1}(\fff)$.  If $N\in \infin$ is such that $N|_n\in M^{-1}(\fff)$ for all $n\in \nn$, then $M(N|_n)\in \fff$ for all $n\in \nn$, contradicting the compactness of $\fff$.  Thus $M^{-1}(\fff)$ is regular.  It is easy to see that for any $0\leqslant \xi<\omega_1$, $M^{-1}(\fff)^\xi = M^{-1}(\fff^\xi)$, so $\iota(M^{-1}(\fff))= \iota(\fff)$. 

(v) Let $F\in M^{-1}(\fff[\mathcal{G}])$.  Then write $M(F)=\cup_{i=1}^n E_i$, where $(E_i)_{i=1}^n\subset \mathcal{G}$ is $\fff$ admissible. Note that for each $1\leqslant i\leqslant n$, $E_i=M(F_i)$ for some $F_i$, which necessarily lies in $M^{-1}(\mathcal{G})$.  Moreover, $M(\min F_i)_{i=1}^n= (\min E_i)_{i=1}^n\in \fff$, and $(\min F_i)_{i=1}^n\in M^{-1}(\fff)$. Note that $F=\cup_{i=1}^n F_i\in M^{-1}(\fff)[M^{-1}(\mathcal{G})]$, so that $M^{-1}(\fff[\mathcal{G}])\subset M^{-1}(\fff)[M^{-1}(\mathcal{G})]$.   

If $E\in M^{-1}(\fff)[M^{-1}(\mathcal{G})]$, write $E=\cup_{i=1}^n E_i$, $(\min E_i)_{i=1}^n\in M^{-1}(\fff)$, $E_i\in M^{-1}(\mathcal{G})$.  Then $(\min M(E_i))_{i=1}^n = M(\min E_i)_{i=1}^n\in \fff$ and $M(E_i)\in \mathcal{G}$. Therefore $M(E)=\cup_{i=1}^n M(E_i)\in \fff[\mathcal{G}]$, and $E\in M^{-1}(\fff[\mathcal{G}])$.  

(vi) Suppose $E\in \mathcal{D}(\mathcal{G}_n)$ and fix $m\leqslant E\in \mathcal{G}_m$.  If $F$ is a subset (resp. spread) of $E$, $m\leqslant F\in \mathcal{G}_m$, so $F\in \mathcal{D}(\mathcal{G}_n)$.  If $N|_m\in \mathcal{D}(\mathcal{G}_n)$ for all $m\in \nn$, then we can choose for each $m\in \nn$ some $k_m\in \nn$ so that $k_m\leqslant N$ and $N|_m\in \mathcal{G}_{k_m}$.  We can, of course, assume that for some $k\leqslant N$, $k_m=k$ for all $m$. Then $N|_m\in \mathcal{G}_k$ for all $m$, a contradiction.  

It is clear that $\iota(\mathcal{D}(\mathcal{G}_n))\geqslant \sup_n \iota(\mathcal{G}_n\cap [n, \infty)^{<\omega})= \sup_n \iota(\mathcal{G}_n)$.  We prove by induction on $\xi< \sup_n \iota(\mathcal{G}_n)$ that $\mathcal{D}(\mathcal{G}_n)^\xi \subset \mathcal{D}(\mathcal{G}_n^\xi)$.  Of course the base case is true. Suppose we have the result for some $\xi<\sup_n \iota(\mathcal{G}_n)$.  Clearly $\varnothing \in \mathcal{D}(\mathcal{G}_n^{\xi+1})$.  If $\varnothing \neq E\in \mathcal{D}(\mathcal{G}_n)^{\xi+1}$, there exists $E\prec F\in \mathcal{D}(\mathcal{G}_n)^\xi\subset \mathcal{D}(\mathcal{G}_n^\xi)$.  Choose $m\leqslant F\in \mathcal{G}_m^\xi$, so that $m\leqslant E\in \mathcal{G}_m^{\xi+1}$.  Therefore $E\in \mathcal{D}(\mathcal{G}_n^{\xi+1})$.  Last, suppose $\xi<\sup_n \iota(\mathcal{G}_n)$ is a limit ordinal.  Clearly $\varnothing \in \mathcal{D}(\mathcal{G}_n^\xi)$.  If $\varnothing \neq E\in \mathcal{D}(\mathcal{G}_n)^\xi$, then we can fix $\xi_m\uparrow \xi$ and $k_m\leqslant E\in \mathcal{G}_{k_m}^{\xi_m}$.  Of course, we can assume $k=k_m$ for all $m\in \nn$, and $E\in \mathcal{G}_k^\xi$.  This proves the claim.  Fix $m\in \nn$ and suppose $\zeta> \max_{1\leqslant n\leqslant m}\iota(\mathcal{G}_n)$.  Then $(m)\notin \mathcal{G}_n^\zeta\cap [n, \infty)^{<\omega}$ for any $n\in \nn$, and $(m)\notin \mathcal{D}(\mathcal{G}_n)^\zeta$.  This proves $\iota(\mathcal{D}(\mathcal{G}_n))\leqslant \sup_n \iota(\mathcal{G}_n)$.  

(vii) First, we observe that for any regular $\fff$, $(\iota(\fff(n)))_{n\in \nn}$ is a non-decreasing sequence.  This is because $\fff(n)$ is homeomorphic to a  subset of $\fff(m)$ for $n\leqslant m$ via the map $E\mapsto (k+m: k\in E)$.  We next observe that if $\iota(\fff)=\xi+1$, then $\iota(\fff(n))=\xi$ eventually.  First, if $\iota(\fff(n))>\xi$ for some $n\in \nn$, then $(n)\in \fff^{\xi+1}$, which means $\iota(\fff)>\xi+1$.  If $\iota(\fff(n))<\xi$ for all $n\in \nn$, then $\fff^\xi$ contains no singletons, and therefore $\iota(\fff)\leqslant \xi$.  

Next, if $\xi$ is a limit ordinal and $\iota(\fff)= \xi$, then $\iota(\fff(n))\nearrow \xi$.  We know $\iota(\fff(n))<\xi$ for all $n\in \nn$ by the same argument as in the successor case.  We know this sequence is non-decreasing, again by the same reasoning as in the successor case.  If $\iota(\fff(n))\leqslant \zeta+1<\xi$ for all $n\in \nn$, then $\iota(\fff)\leqslant \zeta<\xi$.  

Before completing (vii), we complete the following 

\begin{claim} Suppose $\fff, \mathcal{G}$ are regular families with $\iota(\mathcal{G})\geqslant 1$.  Suppose also that for any $n\in\nn$ and any $M\in \infin$, there exist $k_n\in \nn$ and $N\in [M]$ so that $\fff(n)(N)\subset \mathcal{G}(k_n)$.   Then for any $M\in \infin$, there exists $N\in [M]$ so that $\fff(N)\subset \mathcal{G}$.

\end{claim}

\begin{proof}[Proof of claim] If $\iota(\mathcal{G})\geqslant 1$, then for some $k_0$, $\{(k): k\geqslant k_0\}\subset \mathcal{G}$.  Let $M_0=M$ and choose $M_1\in [M_0]$, $k_1\in \nn$ so that $\mathcal{F}(1)(M_1)\subset \mathcal{G}(k_1)$.  By replacing $M_1$ with a subset of $M_1$, we can assume $k_0,k_1\leqslant M_1$.  We can do this since if $M'\in[M_1]$, each member of $\mathcal{F}(1)(M')$ is a spread of $\mathcal{F}(1)(M_1)$, so the desired containment is preserved by passing to $M'$.  

Next, assume that for $1\leqslant i<n$, we have chosen $M_i\in [M_0]$ and $k_i\in \nn$ so that $M_i\in [M_{i-1}]$, $\fff(i)(M_i)\subset \mathcal{G}(k_i)$, and $k_i\leqslant M_i$.  Then choose $k_n\in \nn$ and $M_n\in [M_{n-1}]$ so that $\fff(n)(M_n)\subset \mathcal{G}(k_n)$, and again assume that $k_n\leqslant M_n$.  This completes the recursive choices of $k_n$ and $M_n$.  

Let $M_n=(m^n_i)_i$ and let $N=(m_n^n)$.  Note that $m_1^1<m_2^2<\ldots$ and $k_n\leqslant m_n^n$.  We claim that $\fff(N)\subset \mathcal{G}$.  To see this, fix $E\in \fff$.  If $|E|=0$, $N(E)=\varnothing\in \mathcal{G}$.  If $|E|=1$, then for some $n\in \nn$, $M(E)=(m_n^n)\in \{(k):k\geqslant k_0\}\subset \mathcal{G}$.  Last, if $|E|>1$, we can write $E=n\verb!^!F$ for some $n\in \nn$ and $F\in \fff(n)$.  Since $n<F$, $N(F)$ is a spread of $M_n(F)\in (\fff(n))(M_n)\subset \mathcal{G}(k_n)$.  Therefore $N(F)\in \mathcal{G}(k_n)$, and $k_n\verb!^! N(F)\in \mathcal{G}$.  But since $k_n\leqslant m_n^n$, and since $N(E)=m_n^n\verb!^!N(F)$ is a spread of $k_n\verb!^! N(F)$, $N(E)\in \mathcal{G}$.

\end{proof}

We return to (vii).  If the result were false, we could choose $\zeta<\omega_1$ minimal so that there exists $\eta \leqslant \zeta$ and regular families $\fff, \mathcal{G}$ so that $\iota(\fff)= \eta$, $\iota(\mathcal{G})=\zeta$, and $M\in\infin$ so that for each $N\in [M]$, $\mathcal{F}(N)\not\subset \mathcal{G}$.  Next, we could choose $\xi\leqslant \zeta$ a minimal value of $\eta$ so that the indicated $\fff, \mathcal{G}$, and $M\in \nn$ exist. We assume we have fixed such $\fff, \mathcal{G}, M$.   We consider several cases.  

First, if $\iota(\mathcal{G})=0$, then $\mathcal{G}=\{\varnothing\}=\fff$.  Clearly this cannot be.  

If $\zeta$ is a successor, say $\zeta=\beta+1$, then there exists $n\in \nn$ so that for each $m\geqslant n$, $\iota(\mathcal{G}(m))=\beta$.  If $\xi\leqslant \beta$, then there exists $N\in [M]$ so that $\fff(N)\subset \mathcal{G}(n)\subset \mathcal{G}$, which also cannot be.  Thus if $\zeta=\beta+1$, it must be true that $\xi=\beta+1=\zeta$.  Then for each $m\in \nn$, $\iota(\fff(m))\leqslant \beta$, and by the hypothesis for any $M'\in \infin$ there exist $N'\in [M']$ so that $\fff(m)(N')\subset \mathcal{G}(n)$.   By the claim, we deduce that there exists $N\in [M]$ so that $\fff(N)\subset \mathcal{G}$, and this contradiction means that $\zeta$ cannot be a successor.  

Last, suppose $\zeta$ is a limit ordinal.  Then $\iota(\mathcal{G}(n))\nearrow \zeta$.  If $\xi$ is a successor, then $\xi<\zeta$ and $\iota(\fff(n))\leqslant \xi<\zeta$ for each $n\in \nn$.  If $\xi$ is a limit, then for each $n\in \nn$, by our remarks above, $\iota(\fff(n))<\xi\leqslant \zeta$.  Therefore we can choose a sequence $(k_n)\in \infin$ so that $\iota(\fff(n))\leqslant \iota(\mathcal{G}(k_n))$.  Then by the inductive hypothesis, for $n\in \nn$ and any $M'\in \infin$, there exists $N'\in [M']$ so that $\mathcal{F}(n)(N')\subset \mathcal{G}(k_n)$.  Again, our claim implies that there exists $N\in [M]$ so that $\fff(N)\subset \mathcal{G}$, and this contradiction exhausts the possibilities of ways that (vi) could fail.

(viii) First assume $(\mathcal{G}_\xi)_{0\leqslant \xi<\zeta}$ is either additive or multiplicative.  We prove the first part by induction on $\zeta$.  The $\zeta=\omega$ case is clear, since the families $\mathcal{G}_0\subset \mathcal{G}_1\subset \ldots$ are linearly ordered by inclusion in this case.  Suppose that for a given $\eta<\zeta$ and each $\xi\leqslant \eta$, the conclusion holds.  Suppose $0\leqslant \xi \leqslant \eta+1$.  Then either $\xi=\eta+1$ or $\xi\leqslant \eta$.  In the first case, we can take $m=1$.  In the second case, choose some $m\in \nn$ so that $\mathcal{G}_\xi\cap [m, \infty)^{<\omega}\subset \mathcal{G}_\eta$.  Since $\mathcal{G}_\eta\subset \mathcal{G}_{\eta+1}$, $\mathcal{G}_\xi\cap [m, \infty)^{<\omega}\subset \mathcal{G}_{\eta+1}$.  Last, suppose $\eta<\zeta$ is a limit ordinal and the conclusion holds for each $0\leqslant \xi\leqslant \gamma<\eta$.  Fix $\xi<\eta$ and let $\eta_n\uparrow \eta$ be such that $\mathcal{G}_\eta=\mathcal{D}(\mathcal{G}_{\eta_n})$.  Choose some $n\in \nn$ so that $\xi<\eta_n$ and $k\in \nn$ so that $\mathcal{G}_\xi\cap [k, \infty)^{<\omega}\subset \mathcal{G}_{\eta_n}$.  Let $m=\max\{k,n\}$.  Then $$\mathcal{G}_\xi\cap [m, \infty)^{<\omega}\subset \mathcal{G}_{\eta_n} \cap[n, \infty)^{<\omega}\subset \mathcal{G}_\eta.$$  This completes the first statement in both the additive case and multiplicative cases.

Next, assume $(\mathcal{G}_\xi)_{0\leqslant \xi<\zeta}$ is additive. Observe that if $\mathcal{G}_\xi\cap [m, \infty)^{<\omega}\subset \mathcal{G}_\eta$, then $\mathcal{G}_\xi\cap [m-1, \infty)^{<\omega} \subset (\mathcal{A}_1, \mathcal{G}_\eta)=\mathcal{G}_{\eta+1}$.  By induction, $\mathcal{G}_\xi=\mathcal{G}_\xi\cap [1, \infty)^{<\omega}\subset \mathcal{G}_{\eta+m-1}$.  

Last, assume $(\mathcal{G}_\xi)_{0\leqslant \xi<\zeta}$ is multiplicative.  Observe that $\mathcal{G}_0\subset \mathcal{G}_\xi$ and $(1)\in  MAX(\mathcal{G}_\xi)$ for each $0\leqslant \xi<\zeta$.  We claim that if $\mathcal{G}_\xi \cap [m,\infty)^{<\omega} \subset \mathcal{G}_\eta $ for $m>2$, then $\mathcal{G}_\xi \cap [m-1, \infty)^{<\omega} \subset \mathcal{G}_{\eta+1}$.  This is because if $E=(m-1)\verb!^!F\in \mathcal{G}_\xi [m-1, \infty)^{<\omega}$, then $F\in \mathcal{G}_\xi\cap [m, \infty)^{<\omega} \subset \mathcal{G}_\eta$.  Then $(m-1, \min F)\in \mathcal{S}$, and $E=(m-1)\verb!^!F\in \mathcal{S}[\mathcal{G}_\eta]=\mathcal{G}_{\eta+1}$.  This means that if $\mathcal{G}_\xi\cap [m, \infty)^{<\omega}$, $\mathcal{G}_\xi \cap [2, \infty)^{<\omega} \subset \mathcal{G}_{\eta+ m-2}$.  But since $(1)\in MAX(\mathcal{G}_\xi)\cap \mathcal{G}_{\eta+m-2}$, $\mathcal{G}_\xi = \{(1)\}\cup (\mathcal{G}_\xi\cap [2, \infty)^{<\omega}) \subset \mathcal{G}_{\eta+m-2}$.

\end{proof}

We are now ready to define the fine Schreier families $(\mathcal{F}_\xi)_{0\leqslant \xi<\omega_1}$. These families were defined in \cite{OSZ}, and are a finer version of the more familiar Schreier families defined in \cite{AA}.  We let $\mathcal{F}_0=\{\varnothing\}$.  Next, if $\mathcal{F}_\xi$ has been defined, we let $\mathcal{F}_{\xi+1}=(\mathcal{A}_1, \mathcal{F}_\xi)$.  If $\xi<\omega_1$ is a limit ordinal and $\mathcal{F}_\zeta$ has been defined for each $\zeta<\xi$ so that $(\fff_\zeta)_{0\leqslant \zeta<\xi}$ is additive, fix $\eta_n\uparrow \xi$.  By Proposition \ref{regular facts} (viii), we can choose recursively some natural numbers $m_n$ so that $\fff_{\eta_n+m_n}\subset \fff_{\eta_{n+1}+m_{n+1}}$ for each $n\in \nn$.  We let $\xi_n=\eta_n+m_n$ and let $\mathcal{F}_\xi=\mathcal{D}(\fff_{\xi_n})$. 

We next define the Schreier families, $(\mathcal{S}_\xi)_{0\leqslant \xi<\omega_1}$.  We let $\mathcal{S}_0=\fff_1$,  $\mathcal{S}_{\xi+1}=\mathcal{S}[\mathcal{S}_\xi]$, and if $\mathcal{S}_\zeta$ has been defined for each $\zeta$ less than the countable limit ordinal $\xi$, we fix $\xi_n\uparrow \xi$ and define $\mathcal{S}_\xi=\mathcal{D}(\mathcal{S}_{\xi_n})$.   Proposition \ref{regular facts} and our construction yield the following.

\begin{proposition} For each $0\leqslant \xi<\omega_1$, $\mathcal{F}_\xi$ is regular with $\iota(\mathcal{F}_\xi)=\xi$.  Moreover, for each limit $\xi<\omega_1$, there exists $\xi_n\uparrow \xi$ so that $\fff_\xi=\mathcal{D}(\fff_{\xi_n})$ and so that for each $n\in \nn$, $\fff_{\xi_n}\subset \fff_{\xi_{n+1}}$. For each $0\leqslant \xi<\omega_1$, $\mathcal{S}_\xi$ is regular with $\iota(\mathcal{S}_\xi)=\omega^\xi$.   

\end{proposition}

A straightforward induction proof shows that if $0\leqslant \xi<\omega_1$ and $E\in \fff_\xi'$, then $E\verb!^!(1+\max E)\in \fff_\xi$.  We will implicitly use this fact in our proofs, but it is inessential.

We recall the following dichotomies for subsets of $\fin$.   \begin{theorem}\cite{G} For $\fff, \mathcal{G}\subset \fin$ hereditary, for any $N\in \infin$ there exists $M\in [N]$ so that either $$\fff\cap [M]^{<\omega}\subset \mathcal{G} \text{\ \ or\ \ }\mathcal{G}\cap [M]^{<\omega} \subset \fff.$$  \label{Gasparis}\end{theorem}

\begin{theorem}\cite{PR} For a regular family $\fff$, if $\mathcal{A}, \mathcal{B}\subset MAX(\fff)$ are such that $\aaa\cup \bbb=MAX(\fff)$, then there exists $M\in [\nn]$ so that either $$ MAX(\fff)\cap[M]^{<\omega}\subset \mathcal{A}\text{\ \ or\ \ } MAX(\fff)\cap [M]^{<\omega}\subset \mathcal{B}.$$  

\label{PR}

\end{theorem}

\subsection{The pruning lemmas and applications}

In this section, we discuss two useful lemmas involving prunings.  The notion of a pruning is the regular family analogue of passing to a subsequence of a sequence.  The statement and proof of the pruning lemma require notations which belie the simplicity of the underlying idea, so we say a word about the content before stating it.  Let $\fff\subset \fin$ be a regular family.  For each $E\in \fff'$, suppose that the sequence of immediate successors of $E$ in $\fff$ has a subsequence with some desired property $P_E$ which is allowed to depend on $E$.  Then beginning at the root $\varnothing$ of $\fff$, we can pass to a subsequence of the immediate successors of $\varnothing$ (while ``pruning'' the rest from the tree) so that the remaining sequence has the desired property $P_\varnothing$.  For each immediate successor $E$ of $\varnothing$ which survives the pruning, we pass to a subsequence of the immediate successors of $E$ in $\fff$ which have the desired property $P_E$, and so on.  So, beginning with the root of the tree, we recursively prune the levels of the tree so that in the pruned tree $\mathcal{G}$, for each $E\in \mathcal{G}'$, the sequence of immediate successors of $E$ in $\mathcal{G}$ has the desired property.  All this is done so that, although we have passed to subsequences, $\fff$ and $\mathcal{G}$ have the same ``size.''  

We will say that a function $\phi:\fff\to \fff$ is a \emph{pruning} provided that for each $E\in \fff'$, there exists a strictly increasing function $\psi_E:[s(E), \infty)\to [s(E), \infty)$ so that for each $n\geqslant s(E)$, $\phi(E\verb!^!n) = \phi(E)\verb!^! \psi_E(n)$.  Here, $s(E)=\min \{n\in \nn: E\verb!^!n\in \fff\}$.  The first lemma is essentially contained in \cite{AJO}, so we omit the proof.

\begin{lemma}\cite{AJO} Let $\fff$ be a regular family.  For each $E\in \fff'$, suppose $P_E\in (\fin)^\omega$ is such that some subsequence $(E\verb!^! m)_{m\in M}$ of $(E\verb!^!m)_{m\geqslant s(E)}$ lies in $P_E$.  Then there exists a pruning $\phi:\fff\to \fff$ so that for each $E\in \fff'$, $(\phi(E\verb!^!n))_{n\geqslant s(E)}\in P_{\phi(E)}$.  

\label{pruning lemma}

\end{lemma}

For convenience, in the examples below we freely relabel and denote a pruned tree the same way as the original tree.  In these examples, we will say $(x_E)_{E\in \fff}$ is a weakly null tree (resp. $w^*$ null tree, block tree) if for each $E\in \fff'$, the sequence $(x_{E_n})$ is weakly null (resp. $w^*$ null, a block sequence), where $(E_n)$ is the sequence of immediate successors of $E$ in $\fff$ with the natural enumeration.  

\begin{example} If $X$ is a Banach space with FDD $F$ and $(x_E)_{E\in \widehat{\fff}}\subset X$ is a weakly null tree so that $\inf_{E\in \widehat{\fff}} \|x_E\|=c>0$, then for fixed $\varepsilon>0$, for each $E\in \widehat{\fff}$ we can find $z_E\in c_{00}(F)$ so that $\|z_E\|=\|x_E\|$, $\|x_E-z_E\|< \varepsilon_{|E|}$, and so that for each $E\in \fff'$, $\supp_F(E\verb!^!n)\to \varnothing$.  Here $(\varepsilon_n)\subset (0,1)$ is decreasing to zero at a rate which depends on $c, \varepsilon$, and the projection constant of $F$ in $X$.  If $P_E$ consists of sequences $(E_n)$ of immediate successors of $E$ in $\fff$ so that $(z_{E_n})$ is a seminormalized sequence of successively supported vectors, we can prune to obtain a pruned tree $(y_E)_{E\in \widehat{\fff}}$ of $(x_E)_{E\in \widehat{\fff}}$ and $(u_E)_{E\in \widehat{\fff}}$ of $(z_E)_{E\in \widehat{\fff}}$ so that $\|y_E-u_E\|< \varepsilon_{|E|}$ for each $E\in \widehat{\fff}$ and so that for each $E\in \fff'$, $(u_{E\verb!^!n})$ is a block sequence with respect to $F$.  With an auspicious choice of $(\varepsilon_n)$, for each $E\in \widehat{\fff}$, $(y_{E|_i})_{i=1}^{|E|}$ and $(u_{E|_i})_{i=1}^{|E|}$ will be $(1+\varepsilon)$-equivalent.

\end{example}

\begin{example} Fix a function $f:\fin\to (0,1)$ so that $\sum_{E\in \fin}f(E)<\infty$.  Suppose $g:\fff\to \mathbb{R}$ is any function so that for each $E\in \fff'$, $g(E\verb!^!n)\to 0$.  Then we can find a pruning $\phi:\fff\to \fff$ so that $g(\phi(E))<f(E)$ for each $E\in \fff$.  We will use this in two cases.  

Suppose $\varnothing \neq K\subset X^*$. If $(x_E)_{E\in \widehat{\fff}}\subset B_X$ is such that for each $E\in \widehat{\fff}$ and each $x^*\in K$, $x^*(x_{E\verb!^!n})\to 0$, we say $(x_E)_{E\in \widehat{\fff}}$ is a $K$ \emph{null} tree. Note that if $(c_k)\subset C(\fff)$ is a sequence of pairwise disjoint segments and $(x_k)$ is a sequence so that $x_k$ is a convex combination of $(x_E)_{E\in c_k}$, $(x_k)$ need not be pointwise null on $K$.  We wish to overcome this, which we can easily do under the assumption that $K$ is norm separable.  Let $(x_n^*)$ be a dense sequence in $K$ and let $d(x)=\sum c_n|x^*_n(x)|$, where $(c_n)$ is any sequence of positive numbers so that $\sum c_n \|x_n\|<\infty$.  Note that $(x_n)\subset B_X$ is pointwise null on $K$ if and only if $d(x_n)\to 0$. Suppose that $(x_E)_{E\in \widehat{\fff}}\subset B_X$ is a $K$ null tree, $g(E)=d( x_E)$ for $E\in \widehat{\fff}$, and let $g(\varnothing)=0$.  After pruning, we may assume $d( x_E)< f(E)$ for each $E\in \fff$.  Then suppose $(c_k)_k$ are pairwise disjoint members of $C(\fff)$ and $y_k\in \text{co}(x_E:E\in c_k)\subset B_X$.  Then \begin{align*} \sum_k d(y_k) & \leqslant \sum_k\sum_{E\in c_k} d( x_E) \leqslant \sum_k \sum_{E\in c_k} f(E) \\ & \leqslant \sum_{E\in \fin}f(E)<\infty. \end{align*} Thus $d( y_k)\to 0$, which means $(y_k)$ is pointwise null on $K$.  In the sequel, any $K$ null tree $(x_E)_{E\in \widehat{\fff}}$ in a Banach space $X$ so that any sequence $(x_k)$ with $x_k\in \text{co}\{x_E:E\in c_k\}$, $(c_k)\subset C(\fff)$ pairwise disjoint, is pointwise null on $K$ will be called a \emph{strongly} $K$ \emph{null tree}. In the case that $K=B_{X^*}$, we call a $K$ null tree a \emph{weakly null tree} and a strongly $K$ null tree a \emph{strongly weakly null tree}.

\end{example}

\begin{example}

$(B, d)$ is a metric space and $(b_E)_{E\in \fff}\subset B$ is a tree so that for each $E\in \fff'$, $b_{E\verb!^!n}\to b_E$. We call such a tree a \emph{convergent tree}.  For $E\in \widehat{F}$, let $g(E)=d(b_E, b_{E|_{|E|-1}})$.  Then by passing to a pruning and relabeling, we can assume $d(b_E, b_{E|_{|E|-1}})< f(E)$.  We claim that the resulting tree, which we also denote by $(b_E)_{E\in \fff}$, is such that $E\mapsto b_E$ is continuous. To see this, it is sufficient to show that if $E< E_k$, $k\in \nn$, are such that $\min E_k$ strictly increases and $F_k:=E\verb!^!E_k\in \fff$ for each $k\in \nn$, then $b_{F_k}\to b_E$.  Let $c_k=\{F: E\prec F\preceq E_k\}$, so $(c_k)$ are pairwise disjoint segments.  Therefore \begin{align*} \sum_k d(b_{F_k}, b_E ) & \leqslant \sum_k \sum_{i=|E|+1}^{|F_k|} d(b_{F_k|_i}, b_{F_k|_{i-1}}) \\ & < \sum_k\sum_{F\in c_k} f(F) < \infty.\end{align*} In the sequel, any tree $(b_E)_{E\in \fff}\subset B$ so that $E\mapsto b_E$ is continuous will be called a \emph{continuous tree}. In the case that $B=B_{X^*}$ for some separable Banach space $X$ and $d$ is a metric compatible with the $w^*$ topology on $B_{X^*}$, we refer to these trees as $w^*$ \emph{convergent} and $w^*$ \emph{continuous}, respectively.

\end{example}

\begin{example} Suppose that $X$ is a Banach space and $S, K\subset B_{X^*}$ are norm separable, non-empty sets.  Suppose that $(x_n^*)\subset K-K$ is a $w^*$ null sequence so that $\|x_n^*\|> \varepsilon$ for all $n\in \nn$.  First we can choose for each $n\in \nn$ some $x_n\in B_X$ so that $x^*_n(x_n)> \varepsilon$.  By passing to subsequences, we can assume the sequence $(x_n)$ is pointwise convergent on $S \cup K$.  For $\delta>0$, we can pass to a further subsequences and assume that for any $m<n$, $|x_n^*(x_m)|< \delta$.  Then we let $y_n= (x_{2n}-x_{2n-1})/2$ and $y_n^*= x_{2n}^*$.  Then $y_n^*(y_n)\geqslant \varepsilon/2- \delta/2$ and $(y_n)$ is pointwise null on $S\cup K$.  

Next, suppose $(x^*_E)_{E\in \fff}\subset K-K$ is a $w^*$ null tree so that $\|x^*_E\|>\varepsilon$ for all $E\in \widehat{\fff}$.  We can choose for each $E\in \widehat{\fff}$ some $x_E\in B_X$ so that $x^*_E(x_E)>\varepsilon$.  By using the previous paragraph and pruning, we can assume that for some $\varepsilon'\in (0, \varepsilon/2)$, $(x_E)_{E\in \widehat{\fff}}$ is an $S \cup K$ null tree and $x^*_E(x_E)> \varepsilon'$ for each $E\in \widehat{\fff}$.  Next, we fix decreasing $(\varepsilon_n)\subset (0,1)$ and prune $(x_E)_{E\in \widehat{\fff}}$ using the rule that a sequence $(u_n)$ in $X$ has property $P_E$ provided $|x^*_F(u_n)|< \varepsilon_{|E|+1}$ for all $\varnothing \preceq F\preceq E$ and all $n\in \nn$.  Of course, we pass to the corresponding pruning of $(x^*_E)_{E\in \fff}$.  The result is pair of trees $(x_E)_{E\in \widehat{\fff}}$ and $(x_E^*)_{E\in \fff}$ so that $(x_E)_{E\in \widehat{\fff}}$ is $S\cup K$ null, $(x_E^*)_{E\in \fff}$ is $w^*$ null, and if $\varnothing \preceq E\prec F\in \fff$, $|x^*_E(x_F)|< \varepsilon_{|F|}$.  We last pass to a pruning of $(x_E^*)_{E\in \fff}$ using the rule that a sequence $(u_n^*)$ has property $P_E $ provided $|u_n^*(x_F)|< \varepsilon_{|E|+1}$ for each $\varnothing \prec F\preceq E$.  After passing to the corresponding pruning of $(x_E)$, we have obtained $S$ null and $w^*$ null trees $(x_E)_{E\in \widehat{\fff}}\subset B_X$ and $(x_E^*)_{E\in \fff}\subset K-K$ so that $x^*_E(x_E)\geqslant \varepsilon'$ for each $E\in \widehat{\fff}$ and for each comparable, not equal $E,F$, $|x_E^*(x_F)|< \min\{\varepsilon_{|E|}, \varepsilon_{|F|}\}$.  

Note that this example is also true without the assumption that $K$ is norm separable as long as $X$ does not contain a copy of $\ell_1$.  This is because norm separability was used here to deduce that if $(x_n)\subset B_X$, we can pass to a sequence which is pointwise convergent on $K$.  If $\ell_1$ does not embed into $X$, we can use Rosenthal's $\ell_1$ theorem to pass to a weakly Cauchy subsequence of $(x_n)$, and the rest of the argument goes through unchanged.

 \end{example}

The pruning method defined above is a ``bottom up'' pruning, since it begins at the root of the tree.  We will also want to use a ``top down'' pruning which begins with the leaves of the tree.  

\begin{lemma} Let $K, L$ be compact metric spaces, $\fff$ a regular family, and $k_0:MAX(\fff)\to K$, $l_0:MAX(\fff)\to L$ be any functions.  Then there exist functions $k:\fff\to K$ and $l:\fff\to L$ extending $k_0$ and $l_0$, respectively, and a pruning $\phi:\fff\to \fff$ so that $k\circ \phi, l\circ \phi$ are continuous.  

\label{second pruning lemma}

\end{lemma}

\begin{proof} Recall that for each $E\in \fff'$, we let $s(E)=\min \{n\in \nn: E\verb!^!n\in \fff\}$.   We will define $k(E), l(E)$ for  $E\in MAX(\fff^\zeta)$ by induction on $\zeta$ for $0\leqslant \zeta\leqslant \iota(\fff)$ and $\psi_E:[s(E), \infty)\to [s(E), \infty)$ for $E\in MAX(\fff^\zeta)$ by induction on $\zeta$ for $0<\zeta\leqslant \iota(\fff)$.  Then for $E=(k_1, \ldots, k_n)$, we let $\phi(E)=(\psi_{E|_{i-1}}(k_i))_{i=1}^{|E|}$ so that the resulting tree is convergent.  A second pruning as in the example above will yield a continuous tree.    

For $\xi=0$, we set $k(E)=k_0(E)$, $l(E)=l_0(E)$.  

Next, suppose that for some $\xi$ with $\xi+1\leqslant \iota(\fff)$, $k(E), l(E)$ have been defined for each $E\in \underset{0\leqslant \zeta\leqslant \xi}{\cup}MAX(\fff^\zeta)$ and $\psi_E$ has been defined for each $E\in \underset{1\leqslant \zeta\leqslant \xi}{\cup}MAX(\fff^\zeta)$. Choose $E\in MAX(\fff^{\xi+1})$.  By compactness, we can choose a set $(m^E_n)\in [[s(E), \infty)]$ so that $(k({E\verb!^!m^E_n})), (l({E\verb!^!m^E_n}))$ converge to some $k(E)\in K$, $l(E)\in L$, respectively.  Let $\psi_E(s(E)+n)=m^E_{n+1}$ for $n=0,1,\ldots$.    

Last, suppose that for some limit ordinal $\xi\leqslant \iota(\fff)$, $k(E), l(E)$ have been defined for each $E\in \underset{0\leqslant \zeta<\xi}{\cup}MAX(\fff^\zeta)$ and $\psi_E$ has been defined for each $E\in \underset{1\leqslant \zeta<\xi}{\cup}MAX(\fff^\zeta)$.  The steps in this case are the same as in the successor case.

\end{proof}

\section{Coloring theorems for regular trees} 

If $\xi<\omega_1$ is an ordinal, there exist $k\in \nn$, non-negative integers $n_1, \ldots, n_k$, and $\omega_1>\alpha_1>\ldots >\alpha_k$ so that $$\xi=\omega^{\alpha_1}n_1+\ldots + \omega^{\alpha_k}n_k.$$  If $\xi>0$, there is a unique representation of this form so that each $n_i$ is non-zero.  This is called the \emph{Cantor normal form} of $\xi$.  Let $\xi, \zeta$ be two countable ordinals and $\alpha_1>\ldots>\alpha_k$, $n_i,m_i$ non-negative integers so that $$\xi=\omega^{\alpha_1}m_1+\ldots+\omega^{\alpha_k}m_k$$ and $$\zeta=\omega^{\alpha_1}n_1+\ldots + \omega^{\alpha_k}n_k.$$ By allowing $m_i$ or $n_i$ to be zero, we can assume that the same ordinals $\alpha_i$ are used in the representations of both.  Then we define the \emph{Hessenberg} (or \emph{natural}) \emph{sum} of $\xi$ and $\zeta$ by $$\xi\oplus \zeta = \omega^{\alpha_1}(m_1+n_1)+\ldots + \omega^{\alpha_k}(m_k+n_k).$$  Note that including extra zero terms does not change the value of this sum.  We also note that for each $\xi<\omega_1$, $\{(\alpha, \beta): \alpha\oplus \beta= \xi\}$ is finite.  This sum is not continuous, since $n\oplus n= 2n \to \omega$, while $\omega\oplus \omega = \omega 2$.  But for each $\eta<\omega_1$ and each pair of sequences $(\xi_n), (\zeta_n)$,  $$\sup_n \xi_n\oplus \zeta_n= \omega^\eta \Rightarrow (\sup_n \xi_n)\vee (\sup_n \zeta_n)=\omega^\eta.$$  This is because for natural numbers $n_1, \ldots, n_k$,  $$\omega^\eta > \omega^{\alpha_1}n_1+\ldots + \omega^{\alpha_k}n_k$$ if and only if $\eta>\alpha_1$.  Therefore if $\xi=\sup_n \xi_n, \zeta=\sup_n \zeta_n<\omega^\eta$, $\sup_n \xi_n\oplus \zeta_n \leqslant \xi\oplus \zeta<\omega^\xi.$  Moreover, suppose that $\zeta_m\oplus \eta_m\nearrow \xi$ for a limit ordinal $\xi$.  We can write $$\xi=\omega^{\alpha_1}n_1+\ldots + \omega^{\alpha_k}(n_k+1)$$ for $n_i\geqslant 0$, where $\alpha_k>0$.  Let $\alpha=\omega^{\alpha_1}r_1+\ldots + \omega^{\alpha_k}r_k$ and $\beta=\omega^{\alpha_k}$.  By passing to a subsequence, assume that $\zeta_m\oplus \eta_m=\alpha+\beta_m>\alpha$ for each $m\in \nn$ and note that $\beta_m\nearrow \beta$.  Then for each $m\in \nn$, there exist $s_{1, m}, \ldots, s_{k,m}, t_{1,m}, \ldots, t_{k,m}\geqslant 0$ with $s_{i,m}+t_{i,m}=r_i$ for each $1\leqslant i\leqslant k$ and $\zeta_m', \eta_m'$ so that $\zeta_m'\oplus \eta_m'=\beta_m$, $\zeta_m=\omega^{\alpha_1}s_{1,m}+\ldots + \omega^{\alpha_k}s_{k,m}+\zeta_m'$ and $\eta_m=\omega^{\alpha_1}t_{1,m}+\ldots+\omega^{\alpha_k}t_{1,k}+\eta_m'$.  By our above remarks, either $\zeta_m'\nearrow \beta$ or $\eta_m\nearrow \beta$.  Assume that $\zeta_m'\nearrow \beta$.  By passing to a further subsequence, we can assume that there exist $s_1, \ldots, s_k, t_1, \ldots, t_k$ so that for each $m\in \nn$ and $1\leqslant i\leqslant k$, $s_{i,m}=s_i$ and $t_{i,m}=t_i$.  In this case, with $\zeta''=\omega^{\alpha_1}s_1+\ldots + \omega^{\alpha_k}s_k$ and $\eta''=\omega^{\alpha_1}t_1+\ldots+\omega^{\alpha_k}t_k$, $\zeta_m=\zeta''+ \zeta_m'\nearrow \zeta''+\beta$, $\eta_m\geqslant \eta''$, and $(\zeta''+\beta)\oplus \eta'' = \xi$.  We will use this observation in the limit ordinal case of the proof of our next lemma.

If we give each member of a set $S$ of cardinality $n$ at least one of the two colors $0$ and $1$, of course we can find numbers $i,j$ so that $i+j=n$ and subsets $A,B$ of $S$ with cardinality $i,j$, respectively, so that each member of $A$ gets color $0$, and each member of $B$ gets color $1$.  We wish to generalize this to colorings of regular families, in which case the analogous result, where addition is the Hessenberg sum, is true for colorings of regular families. Here, we consider the case in which each member of $MAX(\fff)$ colors each of its non-empty prececessors with at least one, but possibly both, of the colors $0,1$. If $\fff$ is a regular family, we say a collection $(\mathcal{A}^0_E, \mathcal{A}^1_E)_{E\in \widehat{\fff}}$ of subsets of $MAX(\fff)$ is a \emph{coloring of} $\fff$ if for each $E\in \widehat{\fff}$, $\aaa^0_E\cup \aaa^1_E=\{F\in MAX(\fff): E\preceq F\}$.  

For the sake of simplifying the following proof, we introduce the follwing terminology.  Given regular families $\fff, \mathcal{G}$, we say the pair $(i,e)$ is an \emph{extended embedding} of $\fff$ into $\mathcal{G}$ if $i:\widehat{\fff}\to \widehat{\mathcal{G}}$ is an embedding and $e:MAX(\mathcal{F})\to MAX(\mathcal{G})$ is a function so that for each $E\in MAX(\fff)$, $i(E)\preceq e(E)$.  If $(\aaa^0_E, \aaa^1_E)_{E\in \widehat{\mathcal{G}}}$ is a coloring of $\mathcal{G}$ and $(i,e)$ is an extended embedding of $\fff$ into $\mathcal{G}$, we define for $j=0,1$ and $E\in \widehat{\fff}$ the set $$\mathcal{B}^j_E=\{F\in MAX(\fff): e(F)\in \mathcal{A}^j_{i(E)}\}.$$  We refer to $(\mathcal{B}^0_E, \mathcal{B}^1_E)$ as the \emph{induced coloring of} $\fff$ by $(i,e)$ and $(\mathcal{A}^0_E, \aaa^1_E)$, or, if no confusion can arise, simply the induced coloring.  It is easy to see that this is indeed a coloring of $\fff$.  We say that the induced coloring $(\bbb^0_E, \bbb^1_E)$ is \emph{monochromatically} $j$ provided that for each $E\in MAX(\fff)$, $$e(E)\in \bigcap_{k=1}^{|E|} \mathcal{A}^j_{i(E|_k)}.$$  We observe that if $\eee, \fff, \mathcal{G}$ are regular families, $(\aaa^0_E, \aaa^1_E)_{E\in \widehat{\mathcal{G}}}$ is a coloring of $\mathcal{G}$, $(i,e)$ is any extended embedding of $\eee$ into $\fff$, and if $(i',e')$ is an extended embedding of $\fff$ into $\mathcal{G}$ so that the induced coloring of $\fff$ by $(i',e')$ and $(\mathcal{A}^0_E, \mathcal{A}^1_E)$ is monochromatically $j$, then $(i'\circ i, e'\circ e)$ is an extended embedding of $\eee$ into $\mathcal{G}$ so that the induced coloring of $\eee$ by $(i'\circ i, e'\circ e)$ and $(\aaa^0_E, \aaa^1_E)$ is monochromatically $j$.

\begin{lemma}[Coloring lemma for sums]   Supppose $\fff$ is a regular family with $\iota(\fff)>0$.  If $(\aaa^0_E, \aaa^1_E)$ is a coloring of $\fff$, then for $j=0,1$, there exist an ordinal $\xi_j$ and an extended embedding $(i_j, e_j)$ of $\fff_{\xi_j}$ into $\fff$ so that the induced coloring of $\fff_{\xi_j}$ is monochromatically $j$ and so that $\xi_0\oplus \xi_1=\iota(\fff)$.      

\label{coloring lemma for sums}

\end{lemma}
 
Here, it should be understood that if either $\xi_j=0$ for $j=0$ or $1$, we consider taking $i_j$ and $e_j$ to be the empty maps to satisfy the conclusion for that $j$. 

\begin{proof} We prove the result by induction on $\iota(\fff)$.  Fix $0\leqslant \xi<\omega_1$, and in the case that $\xi>0$ assume the result holds for all families $\fff$ with $\iota(\fff)\leqslant \xi$ and all colorings $(\aaa^0_E, \aaa^1_E)$ of $\fff$.  Fix a regular family $\fff$ with $\iota(\fff)=\xi+1$ and a coloring $(\aaa^0_E, \aaa^1_E)$ of $\fff$.   There exists $n_0\in \nn$ so that for all $n\geqslant n_0$, $\iota(\fff(n))=\xi$.  For each $n\geqslant n_0$, each $E\in \fff(n)$, and $j\in \{0,1\}$, let $$\aaa^j_E(n)=\{F\in MAX(\fff(n)): n\verb!^!F\in \aaa^j_{n\verb!^!E}\}.$$  This defines a coloring of $\fff(n)$, and in fact is the induced coloring on $\fff(n)$ corresponding to the extended embedding $E\mapsto n\verb!^!E$.  Note that for each $n\geqslant n_0$, $\aaa^0_\varnothing(n) \cup \aaa^1_\varnothing(n)=MAX(\fff(n))$.  By Theorem \ref{PR}, there exists $M_n\in \infin$ so that either $$ MAX(\fff(n))\cap [M_n]^{<\omega}\subset \aaa^0_\varnothing(n) \text{\ \ or\ \ }MAX(\fff(n))\cap [M_n]^{<\omega}\subset \aaa^1_\varnothing(n).$$  Without loss of generality, we can find $n_0\leqslant N\in \infin$ so that for each $n\in \nn$, $MAX(\fff(n))\cap [M_n]^{<\omega}\subset \aaa^0_\varnothing(n)$.  Next, for each $n\in N$, choose a function $$f_n:MAX(\fff(n))\to MAX(\fff(n))\cap [M_n]^{<\omega}$$ so that for each $F\in MAX(\fff(n))$, $M_n(F)\preceq f_n(F)$.  We can do this because $\fff(n)$ is regular, which means any member of $\fff(n)\cap [M_n]^{<\omega}$ has an extension in $MAX(\fff(n))\cap [M_n]^{<\omega}$.  Let $(\bbb^0_E(n), \bbb^1_E(n))$ be the coloring on $\fff(n)$ given by $$\bbb^j_E(n)=\{F\in MAX(\fff(n)): f_n(F)\in \mathcal{A}^j_{M_n(E)}(n)\}.$$  It is easy to check that this is indeed a coloring. In fact, this is the induced coloring corresponding to the extended embedding of $\fff(n)$ into itself given by $E\mapsto M_n(E)$ and for $E\in MAX(\fff(n))$, $E\mapsto f_n(E)$.   Now apply the inductive hypothesis to find some $\xi_{0,n}, \xi_{1,n}$ with $\xi_{0,n}\oplus \xi_{1,n}=\iota(\fff(n))=\xi$ and an extended embedding $(i_{j,n}, e_{j,n})$ of $\fff_{\xi_{j,n}}$ into $\fff(n)$ which is monochromatically $j$ with respect to the coloring $(\bbb^0_E(n), \bbb^1_E(n))$.  By passing to a subsequence, we can assume that we have some $n_0\leqslant N\in \infin$ and some $\xi_0, \xi_1$ so that for each $n\in N$, $\xi_{0,n}=\xi_0$ and $\xi_{1,n}=\xi_1$.  By our remarks concerning composing extended embeddings with extended embeddings inducing a monochromatic coloring, for $n\in N$ and $j=0$ or $1$, $$i'_{j,n}(E)=M_n(i_{j,n}(E)), \text{\ \ \ \ \ }e'_{j,n}(E)=f_n(e_{j,n})$$ defines an extended embedding of $\fff_{\xi_j}$ into $\fff(n)$ so that the induced coloring on $\fff_{\xi_j}$ by $(\aaa^0_E(n), \aaa^1_E(n))$ is monochromatically $j$. For convenience, set $i'_{0,n}(\varnothing)=\varnothing$ and let $e'_{0,n}(\varnothing)=\varnothing$ if $\varnothing \in MAX(\fff(n))$.  Define $i_0:\widehat{\fff}_{\xi_0+1}\to \widehat{\fff}$, $e_0:MAX(\fff_{\xi_0+1})\to MAX(\fff)$, $i_1:\widehat{\fff}_{\xi_1}\to \widehat{\fff}$ and $e_1:MAX(\fff_{\xi_1})\to MAX(\fff)$ by $$i_0(k\verb!^! E)= n_k\verb!^! i'_{0, n_k}(E), \hspace{5mm} e_0(k\verb!^!E) = n_k\verb!^!e'_{0,n_k}(E),$$ $$i_1(E)=n_1\verb!^!i'_{1, n_1}(E), \hspace{5mm} e_1(E)=n_1\verb!^! e'_{1, n_1}(E),$$  where $N=(n_k)$.  Again, using our remarks about compositions of extended embeddings, the coloring induced by $(i_1, e_1)$ is monochromatically $1$ with respect to $(\aaa^0_E, \aaa^1_E)$.  To see that the coloring induced by $(i_0, e_0)$ is monochromatically $0$, fix $F\in MAX(\fff_{\xi_0+1})$.  Write $F=k\verb!^!E$.  By our choices and the definition of $(\aaa^0_G(n_k))_G$,  $$e_0(F)=n_k\verb!^!e'_{0, n_k}(E)\in \bigcap_{i=1}^{|E|}\aaa^0_{n_k\verb!^!i'_{0,n_k}(E|_i)}=\bigcap_{i=2}^{|F|} \aaa^0_{i_0(F|_i)}. $$ But by our choices, $e_0(F)\in MAX(\fff_{n_k})\cap [M_{n_k}]^{<\omega}\subset \aaa_{(n_k)}^0$, so $$e_0(F)\in \bigcap_{i=1}^{|F|} \aaa^0_{i_0(F|_i)}.$$  Thus the coloring on $\fff_{\xi_0+1}$ induced by $(i_0, e_0)$ and $(\aaa^0_E, \aaa^1_E)$ is monochromatically $0$.   Since $(\xi_0+1)\oplus \xi_1= \xi_0\oplus \xi_1 + 1 = \xi+1$, this finishes the $\xi+1$ case.

Suppose $\xi$ is a limit ordinal and that the result holds for every coloring of every regular family with $\iota$ index less than $\xi$.  Fix $\fff$ with $\iota(\fff)=\xi$ and a coloring $(\aaa^0_E, \aaa^1_E)$ of $\fff$.  Fix $n_0\in \nn$ so that $(n_0)\in \fff$.  For such $n$, define the coloring $(\aaa^0_E(n), \aaa^1_E(n))$ as was done in the successor case.  Recall that $\iota(\fff(n))\nearrow \xi$.  For each $n\geqslant n_0$, choose $\xi_{j,n}$ so that $\xi_{0,n}\oplus \xi_{1,n}=\iota(\fff(n))$ and extended embeddings $(i_{j,n}, e_{j,n})$ of $\fff_{\xi_j}$ into $\fff(n)$ so that the induced coloring is monochromatically $j$.  Recall by our separation technique that we can pass to a subsequence $N=(n_k)\in \infin$, find ordinals $\alpha, \beta, \beta_k, \gamma$, and find $j\in \{0,1\}$ (which we assume without loss of generality is equal to $0$) so that \begin{enumerate}[(i)]\item $ \xi_{0,n_k}=\alpha+\beta_k$, \item $\beta_k\nearrow \beta$, \item $\beta$ is a limit ordinal, \item $(\alpha+\beta)\oplus \gamma = \xi$, \item $\gamma\leqslant \xi_{1, n_k}$ for all $k\in \nn$. \end{enumerate}  Fix $\zeta_k\uparrow \alpha+\beta$ so that $\fff_{\alpha+\beta}=\mathcal{D}(\fff_{\zeta_k})$ and so that $\fff_{\zeta_k}\subset \fff_{\zeta_{k+1}}$ for all $k\in \nn$.  By passing to a further subsequence of $N$, we can assume without loss of generality that for all $k\in \nn$, $\zeta_k\leqslant \alpha+\beta_k$.  Choose an extended embedding $(i',e')$ of $\fff_\gamma$ into $\fff_{\xi_{1,n_1}}$ and, for each $k\in \nn$, an extended embedding $(i'_k, e'_k)$ of $\fff_{\zeta_k}$ into $\fff_{\alpha+\beta_k}=\fff_{\xi_{0, n_k}}$.  We define extended embeddings $(i_0, e_0)$ and $(i_1, e_1)$ of $\fff_{\alpha+\beta}$ and $\fff_\gamma$, respectively, into $\fff$ so that the coloring induced by $(i_j, e_j)$ is monochromatically $j$ by $$i_1(E)=  n_1\verb!^!(i_{1,n_1} \circ i')(E), \hspace{5mm} e_1(E)=n_1\verb!^!( e_{1,n_1}\circ e')(E)$$ and, if $E\in \widehat{\fff}_{\alpha+\beta}$ with $k=\min E$, $$i_0(E)=n_k\verb!^! (i_{0, n_k}\circ i'_k)(E), \hspace{5mm} e_0(E)=n_k\verb!^!(e_{0, n_k}\circ e'_k)(E).$$

\end{proof}

\begin{lemma}[Coloring lemma for products] Let $\fff, \mathcal{G}$ be regular families.  Suppose $f:C(\mathcal{F}[\mathcal{G}])\to \{0,1\}$ is a function such that for any embedding $j:\mathcal{G}\to \fff[\mathcal{G}]$, there exists $c\in C(j(\mathcal{G}))$ with $f(c)=0$.  Then there exists an order-preserving $j: \widehat{\fff}\to C(\fff[\mathcal{G}])$ so that $f\circ j\equiv 0$.  
\label{coloring products}

\label{Coloring theorem for products}
\end{lemma}

\begin{proof} We first recursively define $r:\widehat{\fff}\to C(\mathcal{G})$ so that for $E\in \widehat{\fff}$, if we let $F_i=\max r(E|_i)\in \mathcal{G}$ for $1\leqslant i\leqslant |E|$, \begin{enumerate}[(i)] \item $(\min F_i)_{i=1}^{|E|}$ is a spread of $E$, hence is a member of $\fff$. \item $(F_i)_{i=1}^{|E|}$ is successive, \item $f\Bigl(\bigl\{\bigl(\cup_{i=1}^{|E|-1} F_i\bigr)\verb!^! F: F\in r(E)\bigr\}\Bigr)=0.$  \end{enumerate} Then $j(E)=\bigl\{\bigl(\cup_{i=1}^{|E|-1}F_i\bigr)\verb!^! F: F\in r(E)\bigr\}$ gives the desired function.  

To perform the base step and inductive step simultaneously, we only need to demonstrate how to perform the construction on the sequence of immediate successors of any $E\in \fff'$.  Suppose that $E\in \fff'$ is such that $r(E|_i)$ has been defined for each $1\leqslant i\leqslant |E|$.  Let $F_i$ be as above.  Let $m_0>E$ be minimal such that $E\verb!^!m_0\in \fff$.  Choose $m_0\leqslant m_1\in \nn$ so that $(\min F_i)_{i=1}^{|E|}\verb!^!m_1\in \fff$. Since $(\min F_i)_{i=1}^{|E|}$ is a spread of $E$, which is non-maximal in $\fff$, such an $m_1$ exists.  If there exists $n\geqslant m_1$ so that $$f\Bigl(\Bigl\{\bigl(\cup_{i=1}^{|E|}F_i\bigr)\verb!^!F: F\in c\Bigr\}\Bigr)=1$$ for all $c\in C(\mathcal{G}\cap (n, \infty)^{\omega})$, we obtain a contradiction.  This is because in this case the embedding $j(G)=\bigl(\cup_{i=1}^{|E|}F_i\bigr)\verb!^! (k+n: k\in G)$ is such that $f|_{j(\mathcal{G})}\equiv 1$.  This is indeed an embedding by our choice of $m_1$ and the fact that $F_i\in \mathcal{G}$ for each $1\leqslant i\leqslant |E|$. We can choose chains $c_{m_0}, c_{m_0+1},\ldots$ so that for each $m\geqslant m_0$, $c_m\in C(\mathcal{G}\cap (m_1, \infty)^{<\omega})$, $\min \min c_m$ is strictly increasing with $m$, and so that $$f\Bigl(\Bigl\{\bigl(\cup_{i=1}^{|E|}F_i\bigr)\verb!^!F: F\in c_m\Bigr\}\Bigr)=0.$$ Setting $r(E\verb!^!m)=c_m$ for each $m\geqslant m_0$ is easily seen to satisfy (i)-(iii).

\end{proof}

\section{The Szlenk and $\ell_1^+$ weakly null indices}

\subsection{Definition and remarks}

Let $X$ be a Banach space and let $L\subset X^*$ be a bounded set.  For $\varepsilon>0$, we let $$s_\varepsilon(L)=\{x^*: \forall w^* \text{\ neighborhoods\ } V \text{\  of\ }x^*, \text{diam}_{\|\cdot\|}(V\cap L)> \varepsilon\}.$$   As usual, we define the transfinite derivatives $$s_\varepsilon^0(L)=L,$$ $$s_\varepsilon^{\xi+1}(L)=s_\varepsilon(s_\varepsilon^\xi(L)),$$ and if $\xi$ is a limit ordinal, $$s_\varepsilon^\xi(L) = \underset{\zeta<\xi}{\bigcap}s_\varepsilon^\zeta(L).$$  It is easy to see that if $L$ is $w^*$ compact, then for each $\xi$, $d^\xi_\varepsilon(L)$ is also $w^*$ compact.  

We define $Sz_\varepsilon(L)=\min\{\xi<\omega_1: s_\varepsilon^\xi(L)=\varnothing\}$ provided this set is non-empty, and $Sz_\varepsilon(L)=\omega_1$ otherwise.  Last, we define $Sz(L)=\sup_{\varepsilon>0}Sz_\varepsilon(L)$.  We define $Sz(X)=Sz(B_X)$.

\begin{proposition}\cite{Sz, La1} Let $X, Y$ be separable Banach spaces, and let $\varnothing \neq K\subset X^*$ be $w^*$ compact. \begin{enumerate}[(i)]\item If $X, Y$ are separable Banach spaces so that $X$ is isomorphic to a subspace of $Y$, then $Sz(X)\leqslant Sz(Y)$. \item If $X$ is a separable Banach space, $Sz(K)<\omega_1$ if and only if $K$ is norm separable. \item If $\varnothing \neq K\subset X^*$ is $w^*$ compact and convex, then either $Sz(K)=\omega_1$ or there exists $\xi<\omega_1$ so that $Sz(K)=\omega^\xi$. \item If $\varnothing \neq K$ is $w^*$ compact, convex, and not norm compact, the supremum $\sup_\varepsilon Sz_\varepsilon(K)$ is not attained. \item $Sz(K)=1$ if and only if $K$ is compact. \end{enumerate} \end{proposition}

\subsection{Weakly null and general $\sigma$ indices} For a given set $S$ and a given $\sigma\subset S^\omega$, we can define the $\sigma$-derivatives and $\sigma$-indices for general hereditary trees on $S$.  Given a tree $\mathcal{H}$ on $S$, we let $$(\mathcal{H})'_\sigma=\{t\in \mathcal{H}: \exists (s_i)\in \sigma, t\verb!^!s_i\in \mathcal{H}\text{\ }\forall i\in \nn\}.$$  If $\mathcal{H}$ is a hereditary tree on $S$, then $(\mathcal{H})'_\sigma$ is also a hereditary tree on $S$.  It is not hard to see that if $\mathcal{H}$ is not hereditary, $(\mathcal{H})'_\sigma$ need not be a tree.  As usual, we define the transfinite $\sigma$-derivatives and $\sigma$-index by $$(\mathcal{H})^0_\sigma= \mathcal{H},$$ $$(\mathcal{H})_\sigma^{\xi+1}= ((\mathcal{H})^\xi_\sigma)'_\sigma,$$ $$(\mathcal{H})^\xi_\sigma=\underset{\zeta<\xi}{\bigcap} (\mathcal{H})^\zeta_\sigma, \xi<\omega_1 \text{\ is a limit ordinal.}$$

We define $I_\sigma(\mathcal{H})=\min \{\xi<\omega_1: (\mathcal{H})^\xi_\sigma=\varnothing\}$ provided this set is non-empty, and $I_\sigma(\mathcal{H})=\omega_1$ otherwise.  We say $\sigma$ \emph{contains diagonals} if any subsequence of a member of $\sigma$ is also a member of $\sigma$, and if for each $j\in \nn$, $(s_{i,j})_i\in \sigma$, then there exists a sequence $(i_j)$ so that $(s_{i_j, j})_j\in \sigma$.  A standard induction proof gives the following.

\begin{proposition}\cite{OSZ} Let $\mathcal{H}$ be a non-empty, hereditary tree on $S$, and suppose $\sigma\subset S^\omega$ contains diagonals.  Then for $0\leqslant \xi<\omega_1$, $I_\sigma(\mathcal{H})>\xi$ if and only if there exists $(t_E)_{E\in \widehat{\fff_\xi}}\subset S$ so that \begin{enumerate}[(i)]\item for each $E\in \widehat{\fff_\xi}$, $(t_{E|_i})_{i=1}^{|E|}\in \mathcal{H}$, \item for each $E\in \fff_\xi'$, $(t_{E\verb!^!n})_{E<n}\in \sigma$. \end{enumerate}

\label{prop5}
\end{proposition}

Observe that in place of $\fff_\xi$, we can use any regular family $\fff$ with $\iota(\fff)=\xi$, since there exists $M\in \infin$ so that $\fff(M)\subset \fff_\xi$ and $\fff_\xi(M)\subset \fff$.  

\begin{example} If $X$ is a Banach space and if $\varnothing \neq K\subset B_{X^*}$ is norm separable, and if $\sigma$ denotes all sequences in $B_X$ which are pointwise null on $K$, then $\sigma$ contains diagonals.  This is because $(x_n)\subset B_X$ is pointwise null on $K$ if and only if $d(x_n)\to 0$, where $d(x)=\sum c_n |x^*(x_n)|$, $(x_n^*)$ is dense in $K$, and $c_n>0$ is chosen so that $\sum c_n \|x_n^*\|<\infty$.  In this case, we denote the \emph{pointwise null on} $K$ \emph{derivative} by $(\mathcal{H})'_K$ and the \emph{pointwise null on} $K$ \emph{index} by $I_K(\mathcal{H})$.  In the case that $K=B_{X^*}$, we refer to this derivative as the the weakly null derivative, denoted $(\mathcal{H})'_w$, and the weakly null index, denoted by $I_w(\mathcal{H})$.

\end{example}

\begin{example} Let $X$ be a Banach space and $\varnothing \neq K\subset X^*$.  For $r>0$, we say $(x_n)\subset B_X$ has $K$ radius $r$ if for any $x^*\in K$, $\lim \sup |x^*(x_n)|\leqslant r$.  If $K$ is norm separable and if $\sigma$ is the collection of sequences $(x_n)\subset B_X$ having $K$ radius $r$, then $\sigma$ contains diagonals.  Clearly any subsequence of a member of $\sigma$ is a member of $\sigma$.  If $(x_n^*)$ is a dense sequence in $K$, and if for each $i\in \nn$, $(x^i_n)_n\in \sigma$, we can choose $i_1, i_2, \ldots$ so that for each $n\in \nn$ and each $1\leqslant k\leqslant n$, $|x^*_k(x^n_{i_n})|< r+1/n$.   Then $(x^n_{i_n})\in \sigma$.  In this case, we let $(\mathcal{H})'_{K,r}$ denote the derivative when $\sigma$ consists of all sequences in $B_X$ with $K$ radius $r$, and $I_{K,r}(\mathcal{H})$ denotes the $\sigma$ index in this case.

\end{example}

\begin{example} If $X$ is a Banach space with FDD $E$, and if $\sigma$ denotes all infinite block sequences in $B_X$ with respect to $E$, then $\sigma$ contains diagonals.  In this case, we denote the block derivative by $(\mathcal{H})_{bl}'$ and the block index by $I_{bl}(\mathcal{H})$.

\end{example}

\begin{example} If $\sigma$ consists of all sequences $(B_n)\subset \fin$ so that $B_n\underset{n}{\to} \varnothing$, then $\sigma$ contains diagonals.  In this case, we also denote the derivative by $(\mathcal{H})_{bl}'$ and the index by $I_{bl}(\mathcal{H})$.  We think of this as a discretized version of the block index for FDDs.

\end{example}

\begin{proposition}\cite{OSZ} Suppose $X$ is a Banach space with FDD $E$.  Let $\mathcal{B}$ be a hereditary block tree in $X$ with respect to $X$.  Let the \emph{compression} $\tilde{\mathcal{B}}$ of $\mathcal{B}$ be defined by $$\tilde{\mathcal{B}}=\{(\max \supp_E(x_i))_{i=1}^k: (x_i)_{i=1}^k\in \mathcal{B}\}.$$  Then for any non-increasing $\overline{\varepsilon}=(\varepsilon_n)\subset (0,1)$, $$\iota(\tilde{\mathcal{B}})\leqslant I_{bl}(\mathcal{B}^{E,X}_{\overline{\varepsilon}}).$$

\label{Zsak}

\end{proposition}

\begin{remark} The compression was defined using minima of supports rather than maxima of supports in \cite{OSZ}, and because of this the result was slightly different.  We include a sketch of the proof to outline how to obtain the version of the statement made here.  

\end{remark}

\begin{proof}[Sketch] First, one defines for any $\mathcal{B}\subset \Sigma(E, X)$ the support tree $$\supp(\mathcal{B})=\{(\supp_E(z_i))_{i=1}^n: (z_i)_{i=1}^n\in \mathcal{B}\}$$ and proves by induction on $\xi$ that for any non-increasing $\overline{\varepsilon}\subset (0,1)$, $$(\supp(\mathcal{B}))_{bl}^\xi \subset \supp\bigl((\mathcal{B}_{\overline{\varepsilon}}^{E,X})_{bl}^\xi\bigr).$$  This part of the proof is unchanged.  

Next, one proves a discretized version of the statement. For each collection $\mathcal{B}$ of finite, successive sequences of finite subsets of $\nn$, one defines $$\max (\mathcal{B})= \{(\max A_i)_{i=1}^n: (A_i)_{i=1}^n\in \mathcal{B}\}.$$  Then one shows by induction that if $\mathcal{B}\subset \fin$ is a hereditary collection of finite, successive sequences of finite subsets of $\nn$, then $$(\max \mathcal{B})^\xi  \subset \max ((\mathcal{B})^\xi_{bl}).$$  Since for any $\mathcal{B}\subset \Sigma(E,X)$, $\tilde{\mathcal{B}}=\max (\supp(\mathcal{B}))$, one applies these two facts to $\tilde{\mathcal{B}}$ to obtain $$\iota(\tilde{\mathcal{B}})=\iota(\max (\supp (\mathcal{B}))) \leqslant I_{bl}(\supp (\mathcal{B}))\leqslant I_{bl}(\mathcal{B}_{\overline{\varepsilon}}^{E,X}).$$  

The difference lies in the discretized version.  If one supposes that \newline $(n_1, \ldots, n_k)\in \max(\mathcal{B})'$ and that $m_j\to \infty$ is such that $(n_1, \ldots, n_k, m_j)\in \max (\mathcal{B})$, we can choose for each $j\in \nn$ some successive $A^j_1, \ldots, A^j_k, C_j\in \fin$ so that $\max A^j_i=n_i$, $\max C_j=m_j$, and $(A^j_1, \ldots, A^j_k, C_j)\in \mathcal{B}$.  Since $A^j_i\subset \{1, \ldots, n_k\}$ for each $j\in\nn$ and each $1\leqslant i\leqslant n$, we can pass to some subsequence and assume we have successive $A_1, \ldots, A_k$ so that $A^j_i=A_i$ for all $j\in \nn$ and $1\leqslant i\leqslant k$. Then $(A_1, \ldots, A_k, C_j)=(A^j_1, \ldots, A^j_k, C_j)\in \mathcal{B}$ for all $j\in \nn$, and $(A_1, \ldots, A_k)\in (\mathcal{B})'_{bl}$. This is how one completes the successor step.  But in the case that we are using minima, we no longer have $A_1^j, \ldots, A^j_k\subset \{1, \ldots, n_k\}$, since we are not controlling the maximum of $A^j_k$, only its minimum.  But if $(n_1, \ldots, n_k)\in (\max (\mathcal{B}))''$, one can fix $m>n_k$ and $m_j\to \infty$ so that $(n_1, \ldots, n_k, m, m_j)\in \mathcal{B}$ for all $j\in \nn$.  One then chooses $A_1^j, \ldots, A^j_k, C_j, D_j$ so that $\min A_i^j=n_i$, $\min C_j= m$, $\min D_j=m_j$.  Now one controls the maximum of $A^j_k$ by controlling the minimum of $C_j$.  In the case of using minima, one must take two Cantor-Bendixson derivatives to one block derivative in order to establish the successor case.  The limit case is clear.

\end{proof}

In what follows, for a Banach space and $\varnothing \neq K\subset B_{X^*}$, we let $$\mathcal{H}^K_\varepsilon = \Bigl\{(x_i)_{i=1}^n\in B_X^{<\omega}: \exists x^*\in K, x^*(x_i)\geqslant \varepsilon \text{\ }\forall 1\leqslant i\leqslant n\Bigr\}.$$  We let $\mathcal{H}^X_\varepsilon = \mathcal{H}^{B_{X^*}}_\varepsilon$.

\subsection{Dualization for separable spaces}

In \cite{AJO}, it was shown that the weakly null $\ell_1^+$ index is equal to the Szlenk index of any separable Banach space not containing $\ell_1$.  Here we discuss how to modify this result to compute the Szlenk index of certain subsets of the dual of a separable Banach space.  

\begin{lemma} If $X$ is a separable Banach space and if $\varnothing \neq K\subset B_{X^*}$ is $w^*$ compact and norm separable, the following are equivalent for any $0<\xi<\omega_1$.  

\begin{enumerate}[(i)] \item There exists $\varepsilon>0$ so that $Sz_\varepsilon(K)>\xi$. 

\item There exists $\varepsilon>0$ so that for every norm separable $\varnothing \neq S \subset X^*$, there exists an $S$ null $(x_E)_{E\in \widehat{\fff_\xi}}\subset B_X$ and a $w^*$ continuous $(x^*_E)_{E\in \fff_\xi}\subset K$ so that for each $E\in \fff_\xi$ and each $\varnothing \prec F\preceq E$, $x_E^*(x_F)\geqslant \varepsilon$.  

\item There exists $\varepsilon>0$ so that for every norm separable $\varnothing \neq S \subset X^*$, there exists an $S$ null $(x_E)_{E\in \widehat{\fff}}\subset B_X$, and $(x_E^*)_{E\in MAX(\fff_\xi)}\subset K$ so that for each $E\in MAX(\fff_\xi)$ and each $\varnothing \prec F\preceq E$, $x_E^*(x_F)\geqslant \varepsilon$. 

\item There exists $\varepsilon>0$ so that for every norm separable $\varnothing \neq S\subset X^*$, $I_S(\mathcal{H}^K_\varepsilon)>\xi$. 

\item There exist $0<r<\varepsilon$ so that for every norm separable $\varnothing \neq S\subset X^*$, $I_{S,r}(\mathcal{H}^K_\varepsilon)>\xi$. \end{enumerate}

Moreover, if $\ell_1$ does not embed into $X$, the result is true without the assumption that $K$ is norm separable.  

\label{dualize}

\end{lemma}

\begin{proof}   \noindent (i)$\Rightarrow$ (ii) Suppose $Sz_\varepsilon(K)>\xi$.  Fix $\varnothing \neq S \subset B_{X^*}$ norm separable.  An easy proof by induction gives that there must exist some tree $(x^*_E)_{E\in \fff_\xi}\subset K$ so that for each $E\in \fff'_\xi$, \begin{enumerate}[(i)] \item $x^*_{E\verb!^!n}\underset{w^*}{\to}x^*_E$, \item $\|x^*_E-x^*_{E\verb!^!n}\|> \varepsilon/2$ for all $n>E$.  \end{enumerate}  For each $E\in \widehat{\fff}_\xi$, let $y_E^*=x^*_E-x^*_{E|_{|E|-1}}$. Let $y^*_\varnothing = x^*_\varnothing$. Then $(y_E^*)_{E\in \widehat{\fff}_\xi}\subset K-K$ is a $w^*$ null tree in $X^*$ so that $\|y^*_E\|> \varepsilon/2$ for all $\varnothing \prec E\in \fff_\xi$.   Fix $0<\delta< \varepsilon'<\varepsilon/4$ and $(\varepsilon_n)\subset (0,1)$ so that for each $n\in \nn$, $\delta>n \varepsilon_n+\sum_{i>n}\varepsilon_i$.  By Lemma \ref{pruning lemma}, we can pass to a pruning and assume $(y_E^*)_{E\in \widehat{\fff_\xi}}$, $(x_E)_{E\in \widehat{\fff}_\xi}\subset B_X$ are such that $(x_E)_{E\in \widehat{\fff}_\xi}$ is $S$ null (actually $S\cup K$ null), $y^*_E(x_E)> \varepsilon'$ for all $E\in \widehat{\fff}_\xi$, and that for each $E\in \widehat{\fff}_\xi$, $F\in \fff_\xi$ comparable and not equal, $$|y^*_F(x_E)|< \min\{\varepsilon_{|E|}, \varepsilon_{|F|}\}.$$  Then for each $E\in \widehat{\fff}$ and $\varnothing \prec F\preceq E$, \begin{align*} \Bigl(y^*_\varnothing + \sum_{i=1}^{|E|} y^*_{E|_i}\Bigr)(x_F) & \geqslant y^*_F(x_F)- \sum_{i=0}^{|F|-1}|y^*_{F|_i}(x_F)| - \sum_{i=|F|+1}^{|E|}|y^*_{E|_i}(x_F)| \\ & > \varepsilon'-|F|\varepsilon_{|F|}- \sum_{i>|F|}\varepsilon_i> \varepsilon'-\delta.\end{align*}  But $$y^*_\varnothing + \sum_{i=1}^{|E|} y^*_{E|_i} = x_\varnothing^* +\sum_{i=1}^{|E|} x^*_{E|_i}-x^*_{E|_{i-1}}=x^*_E\in K.$$  Note that $(x_E^*)_{E\in \fff_\xi}$ is $w^*$ convergent, so by pruning once more and passing to the appropriate pruning of $(x_E)_{E\in \widehat{\fff}_\xi}$ (which is still $S \cup K$ null), we can assume $(x^*_E)_{E\in \fff_\xi}\subset K$ is $w^*$ continuous.

\noindent (ii)$\Rightarrow$ (iii) This is trivial. 

\noindent (iii)$\Rightarrow$ (iv) Suppose $\varepsilon>0$ is such that for each norm separable $\varnothing \neq S \subset X^*$, there exists an $S$ null tree $(x_E)_{E\in \widehat{\fff}_\xi}\subset B_X$ with branches lying in $\mathcal{H}^K_\varepsilon$.  Then in this case, this tree witnesses the fact that $I_S(\mathcal{H}^K_\varepsilon)>\xi$.  

\noindent (iv)$\Rightarrow$ (v) This is trivial, since if $\sigma$ denotes all sequences in $B_X$ pointwise null on $S$ if and $\sigma(r)$ denotes all sequences in $B_X$ with $S$ radius $r$, $\sigma\subset \sigma(r)$.  Thus for any $r>0$, $I_S(\mathcal{H})\leqslant I_{S,r}(\mathcal{H})$ for any $\mathcal{H}$.  

\noindent (v)$\Rightarrow$ (i) We apply (v) with $S=K$.  Suppose $(x_E)_{E\in \widehat{\fff}_\xi}\subset B_X$ is $K$ such that for each $E\in \fff'_\xi$, $(x_{E\verb!^!N})$ has $K$ radius $r$, and so that the branches of this tree lie in $\mathcal{H}^K_\varepsilon$. For each $E\in MAX(\fff_\xi)$< choose $x_E^*\in K$ so that for each $\varnothing \prec F\preceq E$, $x^*_E(x_F)\geqslant \varepsilon$.  Applying Lemma \ref{second pruning lemma} with $L=\{0\}$, we can assume that $(x_E^*)_{E\in \fff_\xi}\subset K$ is $w^*$ continuous.  We claim that for each $0\leqslant \zeta\leqslant \xi$ and any $0<\delta<\varepsilon-r$, $(x^*_E)_{E\in MAX(\fff_\xi^\zeta)}\subset s_{\delta}^\zeta(K)$.  This will yield that $x^*_\varnothing \in s_\delta^\xi(K)$, and $Sz_\delta(K)>\xi$.  

First, note that if $\varnothing \prec E\preceq F$, $x^*_F(x_E)\geqslant \varepsilon$.  To see this, choose any sequence $F\preceq F_n\in MAX(\fff_\xi)$ with $F_n\to F$.  Then because the tree $(x_E^*)_{E\in \fff_\xi}$ is strongly $w^*$ convergent, $$x^*_F(x_E) = \lim x^*_{F_n}(x_E)\geqslant \varepsilon.$$  We now prove the claim by induction on $\zeta$.  The $\zeta=0$ case is clear.  Suppose $E\in MAX(\fff^{\zeta+1}_\xi)$.  Then for each $n>E$, $E\verb!^!n\in MAX(\fff^\zeta_\xi)$.  This means $x^*_{E\verb!^!n}\in s_\delta^\zeta(K)$, whence $x^*_E=w^*\text{-}\lim x^*_{E\verb!^!n}\in s_\delta^\zeta(K)$, by $w^*$ compactness.  But $$\lim \inf\|x^*_E-x^*_{E\verb!^!n}\|\geqslant \lim \inf x^*_{E\verb!^!n}(x_{E\verb!^!n}) - x^*_E(x_{E\verb!^!n})\geqslant \lim\inf x^*_{E\verb!^!n}(x_{E\verb!^!n})-r \geqslant \varepsilon-r>\delta.$$    Here we have used that $(x_{E\verb!^!n})$ has $K$ radius $r$.  This means $x^*_E\in s_{\delta}^{\zeta+1}(K)$.  

Last, suppose $\zeta\leqslant \xi$ is a limit ordinal and that we have the result for all smaller ordinals.  Choose $E\in MAX(\fff^\zeta_\xi)$.  Choose $\zeta_n\uparrow \zeta$ and $E<E_n$ such that $n\leqslant E_n$ and $E\verb!^!E_n\in MAX(\fff^{\zeta_n}_\xi)$.  Then $x^*_{E\verb!^!E_n}\in s_{\delta}^{\zeta_n}(K)$ and $$x^*_E=w^*\text{-}\lim x^*_{E\verb!^!E_n}\in \cap_n s_{\delta}^{\zeta_n}(K)= s^\zeta_{\delta}(K).$$

\end{proof}

\begin{corollary} Let $X$ be a separable Banach space. If $\varnothing \neq K \subset X^*$ is separable, then for any separable $S\supset K$, $Sz(K)=\sup_{\varepsilon>0} I_S(\mathcal{H}^K_\varepsilon).$ If $\ell_1$ does not embed into $X$, then $Sz(K)=\sup_{\varepsilon>0} I_w(\mathcal{H}^K_\varepsilon)$.  

\end{corollary}

\begin{proof} If $\xi<I_S(\mathcal{H}^K_\varepsilon)$ for $S\supset K$ norm separable, or if $I_w(\mathcal{H}^K_\varepsilon)>\xi$, then the proof of (v)$\Rightarrow $ (i) above gives that $Sz(K)>\xi$.  Therefore $Sz(K)\geqslant \sup_{\varepsilon>0}I_S(\mathcal{H}^K_\varepsilon)$, $Sz(K)\geqslant \sup_{\varepsilon>0} I_w(\mathcal{H}^K_\varepsilon)$.  But (i)$\Rightarrow$(iv) above implies that if $\xi<Sz_\delta(K)$, then $\xi< \sup_{\varepsilon>0}I_S(\mathcal{H}^K_\varepsilon)$, and $Sz(K)\leqslant \sup_{\varepsilon>0}I_S(\mathcal{H}^K_\varepsilon)$. The (i)$\Rightarrow $ (iv) in the ``moreover'' case of Lemma \ref{dualize} yields that $Sz(K)\leqslant \sup_{\varepsilon>0}I_w(\mathcal{H}^K_\varepsilon)$ in the case that $\ell_1$ does not embed into $X$.

\end{proof}

\subsection{First application: Minkowski sums}

\begin{theorem} For any separable Banach space $X$, any $\varepsilon>0$, and $\varnothing\neq K, L,S\subset X^*$ norm separable so that $K, L$ are $w^*$ compact, $$I_S(\mathcal{H}^{K+L}_\varepsilon) \leqslant I_S(\mathcal{H}^K_{\varepsilon/2})\oplus I_S(\mathcal{H}^L_{\varepsilon/2}).$$  If $\varnothing \neq K, L\subset X^*$ are also assumed to be convex, then $Sz(K+L)=\max\{Sz(K), Sz(L)\}$.  In particular, for any separable Banach spaces $Y, Z$, $Sz(Y\oplus Z)=\max \{Sz(Y), Sz(Z)\}$. 

\label{minkowski}

\end{theorem}

\begin{remark} The third part of the statement was shown in \cite{OSZ}, where it was shown using slicings of the dual ball.

\end{remark}

\begin{proof}

Let $K_0=K$ and $K_1=L$.  Suppose $\xi<I_S(\mathcal{H}^{K_0+K_1}_\varepsilon)$.  Fix a strongly $S$ null tree $(x_E)_{E\in \widehat{\fff}_\xi}\subset B_X$ with branches lying in $\mathcal{H}^{K_0+K_1}_\varepsilon$.  For each $F\in MAX(\fff_\xi)$, choose $x^*_F(0)\in K_0$ and $x^*_F(1)\in K_1$ so that for each $\varnothing \prec E\preceq F$, $(x^*_F(0)+x^*_F(1))(x_E)\geqslant \varepsilon$.  For $E\in \widehat{\fff}_\xi$ and $j=0$ or $1$, let $$\aaa^j_E =\{F\in MAX(\fff_\xi): E\preceq F, x^*_F(j)(x_E)\geqslant \varepsilon/2\}.$$  By Lemma \ref{coloring lemma for sums}, we can find $\xi_0, \xi_1$ with $\xi_0\oplus \xi_1=\xi$ and for $j=0$ or $1$ an extended embedding $(i_j, e_j)$ of $\fff_{\xi_j}$ into $\fff_\xi$ so that the induced coloring is monochromatically $j$.  But this means that for each $ E\in MAX(\fff_{\xi_j})$, $$e(E)\in \bigcap_{k=1}^{|E|} \aaa^j_{i_j(E|_k)},$$ which means $x^*_{e(E)}(j)(x_{i_j(E|_k)})\geqslant \varepsilon/2$ for $0<k\leqslant |E|$.   Thus the $S$ null tree $(x_{i_j(E)})_{E\in \widehat{\fff_{\xi_j}}}$ witnesses the fact that $\xi_j< I_S(\mathcal{H}^{K_j}_{\varepsilon/2})$.  Then $$\xi=\xi_0\oplus \xi_1 < I_S(\mathcal{H}^K_{\varepsilon/2})\oplus I_S(\mathcal{H}^L_{\varepsilon/2}).$$  Since $\xi<I_S(\mathcal{H}^{K+L}_\varepsilon)$ was arbitrary, $I_S(\mathcal{H}^{K+L}_\varepsilon)\leqslant I_S(\mathcal{H}^K_{\varepsilon/2})\oplus I_S(\mathcal{H}^L_{\varepsilon/2})$.

For the second statement, $\max\{Sz(K), Sz(L)\}=\omega^\xi$ for some $0\leqslant \xi<\omega_1$.  If $\xi=0$, both $K$ and $L$, and therefore $K+L$, must be norm compact.  This gives the result in the case that $\xi=0$.  Suppose $\xi>0$. Then $$I_{K\cup L}(\mathcal{H}^{K+L}_\varepsilon)\leqslant I_{K\cup L}(\mathcal{H}^K_{\varepsilon/2})\oplus I_{K\cup L}(\mathcal{H}^L_{\varepsilon/2})\leqslant I_K(\mathcal{H}^K_{\varepsilon/2})\oplus I_L(\mathcal{H}^L_{\varepsilon/2})<\omega^\xi.$$ Here we have used the fact that $I_K(\mathcal{H}^K_{\varepsilon/2}), I_L(\mathcal{H}^L_{\varepsilon/2})$ are successors, and therefore strictly less than $Sz(K), Sz(L)$, respectively.  Since any sequence pointwise null on $K\cup L$ is pointwise null on $K+L$, we can take the supremum over $\varepsilon$ and deduce $Sz(K+L)\leqslant \max\{Sz(K), Sz(L)\}$. Since $K+L$ contains translates of $K$ and $L$, and since the Szlenk index is translation invariant, $Sz(K+L)\geqslant \max\{Sz(K), Sz(L)\}$.  

For the last part, it is sufficient to assume $Y^*, Z^*$ are separable, since otherwise both sides of the equation are $\omega_1$.  It is clear that $Sz(Y\oplus Z)\geqslant \max\{Sz(Y), Sz(Z)\}$ and that $Sz(Y\oplus Z)=Sz(Y\oplus_1 Z)$, so we assume $Y \oplus Z=Y\oplus_1 Z$. We identify $Y$, $Z$ in the natural way with subspaces of $Y\oplus_1 Z$ and note that with this identifcation, $B_{(Y\oplus_1 Z)^*}= B_{Y^*}+B_{Z^*}$. The previous paragraph now gives the conclusion.

\end{proof}

\subsection{Second application: Szlenk index of an operator}

Given an operator $T:X\to Y$ with $X$ separable, the Szlenk index $Sz(T)$ of $T$ is defined to be $Sz(T^*B_{Y^*})$. The next theorem was shown in \cite{Brooker} using the usual definition of the Szlenk index, while what we show uses our dualization of the Szlenk index.  What we have already done easily yields the following: 

\begin{theorem} For $\xi<\omega_1$, and separable Banach spaces $X, Y$, we let $$\mathcal{SZ}_\xi(X, Y)=\{T\in \mathcal{L}(X, Y): Sz(T)\leqslant \omega^\xi\}.$$  Then for any separable Banach spaces $W, X, Y, Z$, $\xi<\omega_1$, $S\in \mathcal{SZ}_\xi(X, Y)$, and $T\in \mathcal{L}(W,X)$, $R\in \mathcal{L}(Y,Z)$, $RST\in \mathcal{SZ}_\xi(W,Z)$.  Moreover, $\mathcal{SZ}_\xi(X,Y)$ is a closed subspace of $\mathcal{L}(X,Y)$.

\end{theorem}

\begin{proof} Note that $\mathcal{SZ}_0(X,Y)$ is simply the compact operators, so the result is well-known.  Assume $\xi>0$.  We first note that in this case, $S^*B_{Y^*}$ is norm separable and $w^*$ compact for every $S\in \mathcal{SZ}_\xi(X,Y)$.  

Note that $Sz(0)=1$, so $0\in \mathcal{SZ}_\xi$ for any $0\leqslant \xi<\omega_1$.  If $S\in \mathcal{SZ}_\xi$, for any $\varepsilon>0$ and non-zero scalar $c$, $$I_{(cS^*)B_{Y^*}}(\mathcal{H}^{(cS^*)B_{Y^*}}_\varepsilon) = I_{S^*B_{Y^*}}(\mathcal{H}^{(cS^*)B_{Y^*}}_\varepsilon)=I_{S^*B_{Y^*}}(\mathcal{H}^{S^*B_{Y^*}}_{c^{-1}\varepsilon})\leqslant Sz(S).$$  Therefore $Sz(cS)\leqslant Sz(S)$.  Since $c\neq 0$ was arbitrary, $Sz(S)=Sz(cS)$.  

If $Sz(S)>\omega^\xi$, there exists $\varepsilon>0$ so that $I_{S^*B_{Y^*}}(\mathcal{H}^{S^*B_{Y^*}}_\varepsilon)>\omega^\xi$.  This means there exists $(x_E)_{E\in \widehat{\fff}_\xi}\subset B_X$ which is $S^*B_{Y^*}$ null and has branches lying in $\mathcal{H}^{S^*B_{Y^*}}_\varepsilon$.  If $\|S-U\|<\varepsilon/3$, any member of $\mathcal{H}^{S^*B_{Y^*}}_\varepsilon$ is a member of $\mathcal{H}^{U^*B_{Y^*}}_{2 \varepsilon/3}$.  Moreover, any $S^*B_{Y^*}$ null sequence $(x_n)\subset B_X$ is a $U^*B_{Y^*}$ radius $\varepsilon/3$ sequence.  Therefore $(x_E)_{E\in \widehat{\fff}_\xi}\subset B_X$ witnesses the fact that $I_{U^*B_{Y^*}, \varepsilon/3}(\mathcal{H}^{U^*B_{Y^*}}_{2\varepsilon/3})>\xi$.  Therefore $Sz(U)>\omega^\xi$.  Thus $S$ cannot be the norm limit of a sequence lying in $\mathcal{SZ}_\xi$, and $\mathcal{SZ}_\xi$ is a norm closed subset of $\mathcal{L}(X,Y)$.  

Using the fact that $(S^*+U^*)B_{Y^*}\subset S^*B_{Y^*}+U^*B_{Y^*}$ and Theorem \ref{minkowski}, $$Sz(S+U)=Sz((S^*+U^*)B_{Y^*})\leqslant Sz(S^*B_{Y^*}+U^*B_{Y^*})=\max\{S^*B_{Y^*}, U^*B_{Y^*}\},$$ whence $\mathcal{SZ}_\xi(X, Y)$ is closed under finite sums.  

Suppose $S\in \mathcal{SZ}_\xi(X,Y)$ and $R\in \mathcal{L}(Y,Z)$ has norm not exceeding $1$.  Then $$\mathcal{H}^{S^*R^*B_{Z^*}}_\varepsilon\subset \mathcal{H}^{S^*B_{Y^*}}_\varepsilon,$$ since $S^*R^*B_{Z^*}\subset S^*B_{Y^*}$.  Thus $$I_{S^*B_{Y^*}}(\mathcal{H}^{S^*R^*B_{Z^*}}_\varepsilon)\leqslant I_{S^*B_{Y^*}}(\mathcal{H}^{S^*B_{Y^*}}_\varepsilon)\leqslant Sz(S).$$  Since $S^*R^*B_{Z^*}\subset S^*B_{Z^*}$, Lemma \ref{dualize} gives that $$Sz(SR) \leqslant \sup_{\varepsilon>0} I_{S^*B_{Y^*}}(\mathcal{H}^{S^*R^*B_{Z^*}}_\varepsilon) \leqslant \sup_{\varepsilon>0} I_{S^*B_{Y^*}}(\mathcal{H}^{S^*B_{Y^*}}_\varepsilon)=Sz(S)\leqslant \omega^\xi.$$  

Suppose $S\in \mathcal{SZ}_\xi(X,Y)$ and $T\in \mathcal{L}(W,X)$ has norm not exceeding $1$.  Note that $$T(\mathcal{H}^{T^*S^*B_{Y^*}}_\varepsilon)\subset \mathcal{H}^{S^*B_{Y^*}}_\varepsilon. $$ More generally, an easy proof by induction shows that for any $\xi$, $$T((\mathcal{H}^{T^*S^*B_{Y^*}}_\varepsilon)_{T^*S^*B_{Y^*}}^\xi)\subset (\mathcal{H}^{S^*B_{Y^*}}_\varepsilon)_{S^*B_{Y^*}}^\xi.$$  The only non-trivial step is the successor step, for which we note that any sequence $(u_j)\subset B_W$ which is pointwise null on $T^*S^*B_{Y^*}$ is such that $(Tu_j)\subset B_X$ is pointwise null on $S^*B_{Y^*}$.   This proves $Sz(TS)\leqslant Sz(S)$.

\end{proof}

In the next section, we will see a new application of the use of pointwise null indices to computing the Szlenk index of an operator.

\subsection{Third application: Direct sums}

Suppose $(X_n)$ is a sequence of Banach spaces and $U$ is a Banach space with normalized, $1$-unconditional basis $(e_n)$.  We denote by $\Bigl(\oplus_n X_n\Bigr)_U$ the space all sequences $(x_n)$ so that $x_n\in X_n$ and $\sum \|x_n\|e_n\in U$, and let $X$ denote this space with norm $\|(x_n)\|= \bigl\|\sum \|x_n\|e_n\bigr\|$. We also let $P_n:X\to X_n$ denote the operator which takes $(x_m)$ to $x_n$. More generally, for each $E\subset \nn$, we let $P_E=\sum_{n\in E}P_n$.  We have the following.    \begin{enumerate}[(i)] \item $X$ is a Banach space with this norm.  \item $X$ is separable if and only if $X_n$ is separable for each $n\in \nn$. \item If $(e_n)$ is a shrinking basis for $U$, $X^*=\bigl(\oplus_n X_n^*\bigr)_{U^*}$ isometrically.  \item If $(e_n)$ is shrinking, then a sequence $(s_n)\subset X$ is weakly null if and only if it is bounded and $(P_ms_n)_n$ is weakly null in $X_m$ for each $m\in \nn$. \end{enumerate}

\begin{theorem}If $E$ is a Banach space with normalized, $1$-unconditional basis $(e_n)$ and if $X_n$ is a sequence of separable spaces,  $$Sz(X)\leqslant \Bigl(\sup_n Sz(X_n)\Bigr) Sz(U).$$  
\label{infinitedirectsum}
\end{theorem}

\begin{proof} If $U^*$ is non-separable or $X_n^*$ is non-separable for some $n\in \nn$, the result is clear.  Thus it is sufficient to assume that $(e_n)$ is shrinking, which means $X^*$ is separable, and it is sufficient to estimate the $\ell_1^+$ weakly null index.  Let $\xi=\sup_n Sz(X_n^*)$ and $\zeta=I_w(\mathcal{H}^U_{\varepsilon/3})$.  Seeking a contradiction, suppose $I_w(\mathcal{H}^X_\varepsilon)>\xi\zeta$.  Let $(x_E)_{E\in \widehat{\fff_\zeta[\fff_\xi]}}\subset B_X $ be a weakly null tree with branches in $\mathcal{H}^X_\varepsilon$.  Mimicking the proof of Lemma \ref{Coloring theorem for products}, we will recursively construct $r:\widehat{\fff}_\zeta\to C(\fff_\xi)$, $I:\widehat{\fff}_\zeta\to \fin$, $u_E\in B_X$ so that for all $E\in \widehat{\fff}_\zeta$, letting $F_i=\max r(E|_i)$ for each $1\leqslant i\leqslant |E|$ and $F=\cup_{i=1}^{|E|-1} F_i$, \begin{enumerate}[(i)] \item $u_E\in \text{co}(x_{F\verb!^!G}:G\in r(E))$, \item $\|u_E-P_{I(E)}(u_E)\|< 2\varepsilon/3$, \item $(\min F_i)_{i=1}^{|E|}$ is a spread of $E$, \item if $E\prec F\in \fff_\zeta$, $I(E)<I(F)$, \item if $E\verb!^!k, E\verb!^!l\in \fff_\zeta$ with $k<l$, $I(E\verb!^!k)< I(E\verb!^!l)$. \end{enumerate} We then let $j(E)=\{F\verb!^!G: G\in r(E)\}$ to obtain the indicated order preserving function.

For $E\in \fff_\zeta$, we must define $r(E), I(E), u_E$ assuming that $r(F), I(F), u_F$ has been defined for each $\varnothing \prec F\prec E$.  Let $m_0\in \nn$ be minimal so that $E\verb!^!m_0\in \fff_\zeta$.  We will recursively define $r(E\verb!^!m), I(E\verb!^!m), u_{E\verb!^!m}$ for each $m\geqslant m_0$.  Assume that for some $k\geqslant m_0$, these have been defined for each $m_0\leqslant m<k$.  Let $F_i=\max r(E|_i)$ and $F=\cup_{i=1}^{|E|}F_i$.  Fix $n$ so that $F<n$, $I(E)<n$, and $(\min F_i)_{i=1}^{|E|}\verb!^!n\in \fff_\zeta$. This can be done since $(\min F_i)_{i=1}^{|E|}$ is a spread of $E$, which is non-maximal in $\fff_\zeta$.  In the case that $k>m_0$, assume also that $n> I(E\verb!^!(k-1))$. Define $j:\fff_\xi\to \fff_\zeta[\fff_\xi]$ by $$j(G)=F\verb!^!(n+i: i\in G).$$  If for each $c\in C(\fff_\xi\cap(n, \infty)^{<\omega})$, $$\inf \{\|P_{[1, n)}x\|: x\in \text{co}(x_{F\verb!^!G}: G\in c)\}\geqslant \varepsilon/3,$$ then $(P_{[1, n)}x_{j(G)})_{G\in \widehat{\fff}_\xi}\subset B_{\oplus_{i=1}^n X_i}$ is a weakly null tree with branches in $\mathcal{H}^X_{\varepsilon/3}$.  But this would mean that $$\underset{1\leqslant i\leqslant n}{\max} Sz(X_i)=Sz\Bigl(\oplus_{i=1}^n X_i\Bigr)>\xi,$$ a contradiction.  Thus we can find some $c\in C(\fff_\xi\cap (n, \infty))$ so that $$\inf \{\|P_{[1, n)}x\|: x\in \text{co}(x_{F\verb!^!G}): G\in c\}< \varepsilon/3.$$  Let $r(E\verb!^!k)=c$.  Let $u_{E\verb!^!k}\in \text{co}(x_{F\verb!^!G}:G\in c)$ be such that $\|P_{[1, n)}u_{E\verb!^!k}\|< \varepsilon/3$.  Choose $l\in \nn$ so that $\|P_{(l, \infty)}u_{E\verb!^!k}\|< \varepsilon/3$ and let $I(E\verb!^!k)=[n,l]$.  This completes the recursive construction.  

Next, note that since $j$ is order preserving, $(u_{E|_i})_{i=1}^{|E|}$ is a convex block of a member of $\mathcal{H}^X_\varepsilon$, and thus is a member of $\mathcal{H}^X_\varepsilon$.  Let $y_E= \sum_{j\in I(E)}\|P_j (u_E)\|e_j$, so that $$\|y_E\|= \Bigl\|\sum_{j\in I(E)} \|P_j(u_E)\|e_j\Bigr\| = \|P_{I(E)}(u_E)\|\leqslant \|u_E\|\leqslant 1,$$ and, for any $(a_i)_{i=1}^{|E|}\subset [0, \infty)$, \begin{align*} \Bigl\|\sum_{i=1}^{|E|}a_i y_{E|_i}\Bigr\| & = \Bigl\|\sum_{i=1}^{|E|}\sum_{j\in I(E|_i)} a_i\|P_j(u_{E|_i})\|e_j\Bigr\| \\ & = \Bigl\|\sum_j \bigl\|P_j\bigl(\sum_{i=1}^{|E|} a_iP_{I(E|_i)}(u_{E|_i})\bigr)\bigr\|e_j\Bigr\| \\ & = \Bigl\|\sum_{i=1}^{|E|}a_iP_{I(E|_i)}(u_{E|_i})\Bigr\| \\ & \geqslant \Bigl\|\sum_{i=1}^{|E|} a_iu_{E|_i}\Bigr\| - \sum_{i=1}^{|E|}a_i\|u_{E|_i}-P_{I(E|_i)}u_{E|_i}\| \\ & \geqslant (\varepsilon-2\varepsilon/3)\sum_{i=1}^{|E|}a_i = \varepsilon/3\sum_{i=1}^{|E|} a_i. \end{align*}

But $(y_E)_{E\in \widehat{\fff}_\zeta}\subset B_U$ is a block tree, and therefore a weakly null tree. We deduce $I_w(\mathcal{H}^U_{\varepsilon/3})>\zeta$, a contradiction.

\end{proof}

\begin{remark} The result above is optimal in certain cases.  Recall that for $\xi<\omega_1$, the \emph{Schreier space of order} $\xi$, denoted $X_\xi$, is the completion of $c_{00}$ under the norm $$\|x\|_{X_\xi}=\sup_{E\in \sss_\xi}\|P_Ex\|_{\ell_1}.$$  It is known that $Sz(X_\xi)=\omega^{\xi+1}$ \cite{Causey1}.  Fix $\zeta, \xi<\omega_1$ and let $X=(\oplus X_\zeta)_{X_\xi}$.  That is, each member of the sequence of spaces is equal to $X_\zeta$.  Let $(e^n_i)_i$ denote the basis of the space $X_\zeta$ which sits in the $n^{th}$ position in the direct sum.  For $E\in \widehat{\sss_\xi[\sss_\zeta]}$, let $(E_i)_{i=1}^k$ be the standard decomposition of $E$ with respect to $X_\zeta$.  Then let $x_E=e^{\min E_k}_{\max E_k}$.  Then $(x_E)_{E\in \widehat{\sss_\xi[\sss_\zeta]}}\subset B_X$ is weakly null.  Moreover, if $\varnothing \prec E\in \sss_\xi[\sss_\zeta]$ and if $(a_i)_{i\in E}$ are any scalars, then letting $(E_i)_{i=1}^k$ denote the standard decomposition of $E$ with respect to $\sss_\zeta$ and letting $F_i=\cup_{j=1}^i E_j$, \begin{align*} \Bigl\|\sum_{F\preceq E} a_F x_F\Bigr\|_X & \geqslant \sum_{i=1}^k \bigl\|P_{\min E_i} \sum_{F\preceq E} a_Fx_F\bigr\|_{X_\zeta} = \sum_{i=1}^k\bigl\|\sum_{F_{i-1}\prec F\preceq F_i} a_F e^{\min E_i}_{\max F}\|_{X_\zeta} \\ & = \sum_{i=1}^k \sum_{F_{i-1}\prec F\preceq F_i} |a_F|=\sum_{F\preceq E}|a_F|.\end{align*} Thus $Sz(X)>\iota(\sss_\xi[\sss_\zeta])=\omega^{\zeta+\xi}$.  If $\xi$ is infinite, $\zeta+1+\xi+1=\zeta+\xi+1$, so the estimate given by Theorem \ref{infinitedirectsum} of $\omega^{\zeta+\xi+1}$ is optimal in this case.

\end{remark}

\begin{remark} Suppose $U,V$ are Banach spaces with normalized, shrinking, $1$-unconditional bases $(u_n), (v_n)$, respectively, so that the operator $I_{U,V}:U\to V$ defined by $I_{U,V}u_n=v_n$ is bounded.  Suppose that we have two sequences $X_n, Y_n$ of separable Banach spaces and a uniformly bounded sequence of operators $T_n:X_n\to Y_n$.  Then we can define an operator $T:(\oplus X_n)_U\to (\oplus Y_n)_V$ by $T(x_n)=(T_ny_n)$.  An inessential modification of the preceding proof yields that $Sz(T)\leqslant (\sup_n Sz(T_n))Sz(I_{U,V})$.

\end{remark}

\subsection{Fourth application: Constant reduction}

The following argument is a modification of a well-known argument due to James \cite{James}.  Essentially, it is implicitly contained within \cite{AJO}. However, we need more precise quantification than was given there, so we provide a proof.   Suppose $(x_i)_{i=1}^{n^2}\subset B_X$ is such that each convex combination of these points has norm at least $\varepsilon^2$.  We partition $\{1, \ldots, n^2\}$ into successive intervals $I_1<\ldots <I_n$, each having cardinality $n$, and consider two cases.  Either for some $1\leqslant i\leqslant n$, all convex combinations of $(x_j)_{j\in I_i}$ have norm at least $\varepsilon$, or for each $1\leqslant i\leqslant n$, we can find a convex combination $y_i=\sum_{j\in I_i}a_jx_j$ of $(x_j)_{j\in I_i}$ so that $\|y_i\|<\varepsilon$.  Then $(\varepsilon^{-1} z_i)_{i=1}^n\subset B_X$, by homogeneity, has the property that each convex combination of this sequence has norm at least $\varepsilon$.  Thus in either case, we have found in $B_X$ a multiple of a convex block of $(x_i)_{i=1}^{n^2}$ having length $n$ and so that each convex combination of the resulting sequence has norm at least $\varepsilon$.  

Below, we view a tree of order $\xi^2$ as being composed of a tree of order $\xi$, with vertices each being a tree of order $\xi$.  We will again consider two cases: One of these ``interior'' trees will already have the lower $\varepsilon$ estimate on all of its branches, or we can replace each of these trees with a ``bad'' convex combination so that, after being multiplied by $\varepsilon^{-1}$, these ``bad'' combinations will form a tree of size $\xi$ having the appropriate $\varepsilon$ lower estimates on all convex combinations of all branches.

\begin{theorem} For $\delta,\varepsilon\in (0,1)$ and a Banach space $X$ having separable dual, $$I_w(\mathcal{H}^X_{\varepsilon \delta})\leqslant I_w(\mathcal{H}^X_\varepsilon)I_w(\mathcal{H}^X_\delta).$$  If $I_w(\mathcal{H}^X_\varepsilon)>\omega^{\omega^\xi}$ for some $\xi$, then, $I_w(\mathcal{H}^X_\delta)>\omega^{\omega^\xi}$.  In particular, if $\eta<\omega_1$ is a limit ordinal, $Sz(X)\neq \omega^{\omega^\eta}$. 

Moreover, if $\eta<\omega_1$ is any limit ordinal, and if $Y$ is any Banach space, $Sz(Y)\neq \omega^{\omega^\eta}$.   \label{constant reduction}
\end{theorem}

\begin{proof} Let $\xi=I_w(\mathcal{H}^X_\varepsilon)$. Fix $0<\zeta<\omega_1$. Assume that $I_w(\mathcal{H}^X_{\varepsilon\delta})>\xi\zeta$.  Then we can find a strongly weakly null tree $(x_E)_{E\in\widehat{\fff_\zeta[\fff_\xi]}}\subset B_X$ the branches of which lie in $\mathcal{H}^X_{\varepsilon\delta}$.  We define a coloring on $C(\widehat{\fff_{\zeta}[\fff_\xi]})$ by letting $f(c)=0$ provided there exists a convex combination of $(x_E)_{E\in c}$ which has norm less than $\varepsilon$, and color $1$ otherwise.  If there exists an embedding $i:\widehat{\fff}_\xi\to \widehat{\fff_\zeta[\fff_\xi]}$ so that each $c\in C(i(\widehat{\fff}_\xi))$ receives color $1$, then $(x_{i(E)})_{E\in \widehat{\fff}_\xi}$ witnesses the fact that $I_w(\mathcal{H}_\varepsilon^X)>\xi$, a contradiction.  Therefore for each embedded tree $i(\widehat{\fff}_\xi)$, some branch of this embedded tree receives color $0$.  Applying Lemma \ref{Coloring theorem for products}, we obtain an order preserving $j:\widehat{\fff}_\zeta\to C(\fff_\zeta[\fff_\xi])$ so that for each $E\in \widehat{\fff}_\zeta$, $j(E)$ receives color $0$.  Letting $y_E$ be a convex combination of $(x_F)_{F\in j(E)}$ with norm less than $\varepsilon$, we obtain a weakly null tree $(y_E)_{E\in \widehat{\fff}_\zeta}$. This tree is weakly null because the original tree was strongly weakly null.  Since $j$ is order preserving, $(y_{E|_i})_{i=1}^{|E|}$ is a convex block of a member of $\mathcal{H}^X_{\varepsilon\delta}$, and therefore lies in $\mathcal{H}^X_{\varepsilon\delta}$.  Then by homogeneity, $(\varepsilon^{-1}y_E)_{E\in \widehat{\fff}_\zeta}\subset B_X$ is a weakly null tree with branches in $\mathcal{H}^X_\delta$.  This means $I_w(\mathcal{H}^X_\delta)>\zeta$, which proves the first inequality.    

Suppose $I_w(\mathcal{H}^X_\varepsilon)>\omega^{\omega^\xi}$ for some $\xi$.  Fix $\zeta<\omega^{\omega^\xi}$.  Choose $n\in \nn$ so that $\varepsilon^{1/n}>\delta$.  Note that $\zeta^n<\omega^{\omega^\xi}$, so $I_w(\mathcal{H}^X_\varepsilon)>\zeta^n$.  By applying the first inequality, we deduce $I_w(\mathcal{H}^X_\delta)>\zeta$.  Since $\zeta<\omega^{\omega^\xi}$ was arbitrary, $I_w(\mathcal{H}_\delta^X)\geqslant \omega^{\omega^\xi}$. But since $I_w(\mathcal{H}^X_\delta)$ is always a successor, $I_w(\mathcal{H}^X_\delta)>\omega^{\omega^\xi}$.  

If $Sz(X)=\omega^{\omega^\eta}$, then $X^*$ is separable.  This means that for $\zeta<\eta$, $I_w(\mathcal{H}^X_\varepsilon)>\omega^{\omega^\zeta}$ for some $\varepsilon\in (0,1)$, and by the preceding part, $I_w(\mathcal{H}^X_{1/2})> \omega^{\omega^\beta}$.  But since this holds for any $\zeta<\eta$, $I_w(\mathcal{H}^X_{1/2})\geq \sup_{\zeta<\eta} \omega^{\omega^\zeta}=\omega^{\omega^\eta}$.  Again, since $I_w(\mathcal{H}^X_{1/2})$ is a successor, this must be a strict inequality, which means $Sz(X)>\omega^{\omega^\eta}$.  

For the last statement, we cite a result of Lancien \cite{La1} which states that if the Szlenk index of a Banach space is countable, it is separably determined.  Therefore if there existed a Banach space $Y$ with $Sz(Y)=\omega^{\omega^\eta}$, $\eta$ countable, then $Y$ would have a separable subspace $X$ with $Sz(X)=\omega^{\omega^\eta}$. But this means $X^*$ is separable, which means $Sz(X)=\omega^{\omega^\eta}$ is impossible.  

\end{proof}

\subsection{Fifth application: Three-space properties}

Given our dualization lemma, the following theorem can be shown to be equivalent to Proposition $2.1$ of \cite{BL} in the case of a Banach space having separable dual, up to the value of certain constants.  There, however, the result was shown using the usual definition of Szlenk index involving slicing the dual ball, whereas we use only the weakly null $\ell_1^+$ index.

\begin{theorem} For any $\varepsilon\in (0, 1/3)$, any Banach space $X$ having separable dual, and any closed subspace $Y\leqslant X$, $$I_w(\mathcal{H}^X_\varepsilon)\leqslant I_w(\mathcal{H}^{X/Y}_{\varepsilon/5})I_w(\mathcal{H}^Y_{\varepsilon/5}).$$  In particular, for any ordinal $\xi<\omega_1$, $Sz(\cdot)\leqslant \omega^{\omega^\xi}$ and $Sz(\cdot)< \omega^{\omega^\xi}$ are three space properties on the class of separable Banach spaces.  

\end{theorem}

\begin{proof} Fix a Banach space $X$ having separable dual, $\varepsilon\in (0,1/3)$, and $Y\leqslant X$.  Let $Q:X\to X/Y$ denote the quotient map. Let $\xi=I_w(\mathcal{H}^{X/Y}_{\varepsilon/5})$ and $\zeta=I_w(\mathcal{H}^Y_{\varepsilon/5})$.  If $I_w(\mathcal{H}^X_\varepsilon)>\xi\zeta$, we can find a strongly weakly null tree $(x_E)_{E\in \widehat{\fff_\zeta[\fff_\xi]}}\subset S_X$ with branches in $\mathcal{H}^X_\varepsilon$.  Define the coloring $f$ on $C(\fff_\zeta[\fff_\xi])$ by letting $f(c)=0$ provided that for each convex combination $x$ of $(x_E)_{E\in c}$, $\|Qx\|_{X/Y}\geqslant \varepsilon/5$.  If there exists an embedding $i:\widehat{\fff}_\xi\to \widehat{\fff_\zeta[\fff_\xi]}$ so that $f(c)=1$ for all $c\in C(i(\widehat{\fff}_\xi))$, then $(Qx_{i(E)})_{E\in \widehat{\fff}_\xi}\subset B_{X/Y}$ is a weakly null tree witnessing the fact that $I_w(\mathcal{H}^{X/Y}_{\varepsilon/5})>\xi$, a contradiction.  Therefore we apply Lemma \ref{Coloring theorem for products} to obtain an order preserving $j:\widehat{\fff}_\zeta\to \widehat{\fff_\zeta[\fff_\xi]}$ so that $f\circ j\equiv 0$.  For each $E\in \widehat{\fff}_\zeta$, we let $z_E$ be a convex combination of $(x_F)_{F\in j(E)}$ so that $\|Qz_E\|_{X/Y}<\varepsilon/5$.  For each $E\in \widehat{\fff}_\zeta$, $(z_{E|_i})_{i=1}^{|E|}$ is a convex block of a member of $\mathcal{H}^X_\varepsilon$, and is therefore also a member of $\mathcal{H}^X_\varepsilon$. 

For $E\in \fff_\zeta'$, Proposition \ref{perturb} gives that there exists a subsequence $(z_{E\verb!^!k_n})$ of $(z_{E\verb!^!n})_{E<n}$, and a weakly null sequence $(y_n)_{E<n}\subset B_Y$ so that $\|z_{E\verb!^!n}-y_n\|<4\varepsilon/5$.  By Lemma \ref{pruning lemma}, we can find a pruned subtree $(u_E)_{E\in \widehat{\fff}_\zeta}$ of $(z_E)_{E\in \widehat{\fff}_\zeta}$ and a weakly null tree $(y_E)_{E\in \widehat{\fff}_\zeta}\subset B_Y$ so that for each $E\in \widehat{\fff}_\zeta$, $\|u_E-y_E\|<4\varepsilon/5$.  For each $E\in \widehat{\fff}_\xi$, since $(u_{E|_i})_{i=1}^{|E|}\in \mathcal{H}^X_\varepsilon$, there exists $f\in B_{X^*}$ so that $f(u_{E|_i})\geqslant \varepsilon$ for each $1\leqslant i\leqslant |E|$.  Then for such $i$, $f(y_{E|_i})\geqslant f(u_{E|_i})- \|u_E-y_E\| >\varepsilon-4\varepsilon/5=\varepsilon/5$.  Thus $(y_E)_{E\in \widehat{\fff}_\zeta}\subset B_Y$ witnesses the fact that $I_w(\mathcal{H}^Y_{\varepsilon/5})>\zeta$, a contradiction.  This proves the first statement.

For the second and third parts, assume $Sz(Y), Sz(X/Y)\leqslant \omega^{\omega^\xi}$.  This means $Y^*, (X/Y)^*$, and therefore $X^*$, are separable.  Moreover, for each $\varepsilon\in (0,1/3)$, $$I_w(\mathcal{H}^X_\varepsilon)\leqslant I_w(\mathcal{H}^{X/Y}_{\varepsilon/5})I_w(\mathcal{H}^Y_{\varepsilon/5})<\omega^{\omega^\xi},$$ since $I_w(\mathcal{H}^{X/Y}_{\varepsilon/5}),I_w(\mathcal{H}^Y_{\varepsilon/5})<\omega^{\omega^\xi}$.  Since this holds for all $\varepsilon$, $Sz(X)\leqslant \omega^{\omega^\xi}$.  Moreover, if $Sz(Y)$, $Sz(X/Y)<\omega^{\omega^\xi}$, then $Sz(X/Y)Sz(Y)<\omega^{\omega^\xi}$, and $$\sup_{\varepsilon\in (0,1/3)}I_w(\mathcal{H}^X_\varepsilon) \leqslant Sz(X/Y)Sz(Y)<\omega^{\omega^\xi}.$$

\end{proof}

\section{Classes of Banach spaces with bounded Szlenk index}

\subsection{Mixed Tsirelson spaces}

For our purposes, mixed Tsirelson spaces are a remarkably useful class of spaces for providing examples with prescribed $\ell_1$ behavior.  For example, given a sequence of countable ordinals $\xi_n\nearrow \omega^\xi$ and constants $1\geqslant \theta_n\searrow 0$, does there exist a Banach space $X$ so that $\omega^\xi>I_w(\mathcal{H}^X_{\theta_n})\geqslant \xi_n$ for each $n\in \nn$?  Theorem \ref{constant reduction} says this is not possible for arbitrary sequences, since $I_w(\mathcal{H}^X_{\theta^n}) \leqslant I_w(\mathcal{H}^X_\theta)^n$ for any $\theta\in (0,1)$.   In the cases when this estimate is essentially optimal, i.e. the cases when we have roughly geometric growth, we encounter this restriction.  This restriction is the only one, however, as the mixed Tsirelson spaces show.  For this, we will use the mixed Tsirelson spaces.

Let $(e_n)$ denote the canonical $c_{00}$ basis and let $P_n$, $P_E$ denote the associated canonical coordinate and partial sum projections.  Suppose that $1> \theta_n\searrow 0$ and $(\mathcal{G}_n)_{n\geqslant 0}$ are regular families so that $\mathcal{G}_0$ contains all singletons.  Define the norm $\|\cdot\|_{\mathcal{G}_0}$ on $c_{00}$ by $$\|x\|_{\mathcal{G}_0}=\max_{E\in \mathcal{G}_0} \|P_Ex\|_{\ell_1}.$$   We inductively define norms $|\cdot|_n$, $n=0,1,2, \ldots$ on $c_{00}$ by $|x|_0=\|x\|_{\mathcal{G}_0}$ and $$|x|_{n+1} = |x|_n \vee \sup_{m\in \nn}\sup\Bigl\{ \theta_m\sum_{i=1}^k |P_{E_i}x|_n: (E_i)_{i=1}^k \text{\ is\ }\mathcal{G}_m\text{\ admissible}\Bigr\}.$$  One can easily prove by induction that $|x|_n\leqslant \|x\|_{\ell_1}$ so that $\|x\|=\sup_n |x|_n$ is a well-defined norm on $c_{00}$ making the canonical $c_{00}$ basis normalized and $1$-unconditional satisfying the implicit equation $$\|x\|=\|x\|_{\mathcal{G}_0}\vee \sup_{m\in \nn} \sup\Bigl\{\theta_m\sum_{i=1}^k \|P_{E_i}x\|: (E_i)_{i=1}^k \text{\ is\ }\mathcal{G}_m\text{\ admissible}\Bigr\}.$$  We let $T(\mathcal{G}_0, (\theta_n, \mathcal{G}_n))$ denote the completion of $c_{00}$ with respect to this norm.  In the special case that this space is built from a single family $\mathcal{G}$ and a single constant $\theta\in (0,1)$, we denote the resulting space by $T(\theta, \mathcal{G})$.  This is the case where $\mathcal{G}_0=\mathcal{S}_0$, $\mathcal{G}_n=[\mathcal{G}]^n$, and $\theta_n=\theta^n$.  This space coincides with the usual Tsirelson space $T_{\xi, \theta}$ when $\mathcal{G}=\mathcal{S}_\xi$, and is isomorphic to either $c_0$ or $\ell_p$ for some $p>1$ in the case that $\mathcal{G}=\fff_n$ for some $n\in \nn$ \cite{AT}.  We will use the following results. Item (ii) comes from \cite{LT} and (iii) comes from \cite{JO}.      
\begin{proposition}Fix regular families $(\mathcal{G}_n)_{n\geqslant 0}$ so that $\mathcal{G}_0$ contains all singletons and constants $1>\theta_n\to 0$.  Let $T=(\mathcal{G}_0, (\theta_n, \mathcal{G}_n))$. \begin{enumerate}[(i)] \item For any $0\leqslant k$, $m\in \nn$, $I_w(\mathcal{H}_{\theta_m^k}^T)\geqslant \iota(\mathcal{G}_0)\iota(\mathcal{G}_m)^k$. \item If $\iota(\mathcal{G}_0)\geqslant \sup_n \iota(\mathcal{G}_n)^\omega$, $Sz(T)=\iota(\mathcal{G}_0)\sup_n\iota(\mathcal{G}_n)^\omega$. \item For any $\theta\in (0,1)$ any $\xi<\omega_1$, and any $M\in \infin$, $Sz(T(\theta,M^{-1}( \mathcal{S}_\xi)))=\omega^{\xi\omega}$. \item For any $\theta\in (0,1)$ and any $n\in \nn$, $Sz(T(\theta, \fff_n))=\omega$.  \end{enumerate}

\label{tsirelson facts}

\end{proposition}

\begin{proof}(i) One can easily show by induction on $k$ that if $E\in [\mathcal{G}_n]^k[\mathcal{G}_0]$, then for any scalars $(a_i)_{i\in E}$, $$\Bigl\|\sum_{i\in E} a_ie_i\Bigr\|\geqslant \theta_n^k\sum_{i\in E} |a_i|.$$  Once we establish that the basis of $T$ is shrinking, which we will do below, this will give that $(e_{\max E})_{E\in \widehat{[\mathcal{G}_n]^k[\mathcal{G}_0]}}$ is a normalized, weakly null tree with branches in $\mathcal{H}^T_{\theta_n^k}$.  This guarantees that $I_w(\mathcal{H}^T_{\theta_n^k})> \iota([\mathcal{G}_n]^k[\mathcal{G}_0])= \iota(\mathcal{G}_0)\iota(\mathcal{G}_n)^k$.  

For (ii)-(iv), we must first define the Bourgain $\ell_1$ block index of a basis, first introduced in \cite{Bo}.  Given a Banach space $X$ with basis $(e_i)$, for $K\geqslant 1$, we let \begin{align*} \mathcal{T}(X, (e_i), K)=\Bigl\{(x_i)_{i=1}^n\in \Sigma((e_i), X)& : n\in \nn, \text{\ }\forall (a_i)_{i=1}^n\subset \mathbb{R}\\ & K\Bigl\|\sum_{i=1}^n a_ix_i\Bigr\|\geqslant \sum_{i=1}^n |a_i| \Bigr\}.\end{align*}  With the derivations $d^\xi$ and order $o$ as defined in Section $3.1$, we define $$B(X, (e_i),K)= o(\mathcal{T}(X, (e_i), K)), \text{\ \ }B(X, (e_i))=\sup_{K\geqslant 1} B(X, (e_i), K).$$  We recall that $\ell_1$ embeds into $X$ if and only if $B(X, (e_i))=\omega_1$.  Moreover, if $(e_i)$ is $1$-unconditional, and $I_w(\mathcal{H}^X_\varepsilon)> \xi$, we can replace $\varepsilon$ with any strictly smaller number $\delta$ and use a standard perturbation argument to find a block tree $(x_E)_{E\in \widehat{\fff}_\xi}\subset B_X$ with branches in $\mathcal{H}^X_\delta$.  By $1$-unconditionality, for all $E\in \widehat{\fff}_\xi$ and scalars $(a_i)_{i=1}^{|E|}$, $$\delta^{-1} \Bigl\|\sum_{i=1}^{|E|}a_ix_{E|_i}\Bigr\|\geqslant \sum_{i=1}^{|E|}|a_i|.$$  One then shows by induction that for each $0\leqslant \zeta\leqslant \xi$, $$(x_E)_{E\in \widehat{\fff^\zeta_\xi}}\subset d^\zeta(\mathcal{T}(X, (e_i), \delta^{-1})),$$ whence $$I_w(\mathcal{H}^X_\varepsilon) \leqslant B(X, (e_i), \delta^{-1})\leqslant B(X, (e_i)).$$  By \cite{LT}, $B(T, (e_i))<\omega_1$, so that $\ell_1$ does not embed into $T$ for any choice $(\mathcal{G}_n)_{n\geqslant 0}$, $1>\theta_n\to 0$. By \cite{JO}, $B(T(\theta, \mathcal{S}_\xi))=\omega^{\xi\omega}$.  Since $T(\theta, \fff_n)$ is isomorphic to either $c_0$ or $\ell_p$ for some $p>1$, we deduce that none of these spaces contains $\ell_1$, and that the basis of each is shrinking.  For (ii) and (iv), it remains to note that $B(T, (e_i))=\iota(\mathcal{G}_0)\sup_{k,n}\iota(\mathcal{G}_n)^k$ \cite{LT}, and $B(\ell_p, (e_i))=B(c_0, (e_i))=\omega$ for $p>1$.  For (iii), we note that $Sz(T(\theta, M^{-1}(\mathcal{S}_xi)))\geqslant \omega^{\xi\omega}$ by (i).  It is easy to see that the sequence $(e_n)$ in $T(\theta, M^{-1}(\mathcal{S}_\xi))$ is isometrically equivalent to $(e_{m_n})$ in $T(\theta, \mathcal{S}_\xi)$ by proving by induction that they are isometrically equivalent with respect to each norm $|\cdot|_n$ in the definitions of these spaces.  Therefore $$Sz(T(\theta, M^{-1}(\mathcal{S}_\xi)))\leqslant B(T(\theta, M^{-1}(\mathcal{S}_\xi)), (e_i))\leqslant B(T(\theta, S_\xi), (e_i))=\omega^{\xi\omega}.$$

\end{proof}

With this, we arrive at a characterization of the countable ordinals which occur as the Szlenk index of a Banach space.  We note that in \cite{LT}, the corresponding result for the Bourgain $\ell_1$ index was established, and the result below only requires a minor modification of their result combined with Lancien's result in \cite{La1} that the Szlenk index, when countable, is separably determined.

\begin{theorem} Let $1\leqslant \xi<\omega_1$ be an ordinal.  The following are equivalent: \begin{enumerate}[(i)]\item There exists a Banach space $X$ with $Sz(X)=\omega^\xi$. \item There exists a mixed Tsirelson space $T$ with $Sz(T)=\omega^\xi$. \item There does not exist a limit ordinal $\zeta$ so that $\xi=\omega^\zeta$.  \end{enumerate}

\label{existence theorem}

\end{theorem}

\begin{proof}   We consider several cases.  

Case $1$: $\xi=0$. Then for any finite dimensional $X$, $Sz(X)=1=\omega^0$.  

Case $2$: $\xi=1$. Then $Sz(T(1/2, \fff_1))=\omega$.  

Case $3$: $\xi=\omega^{\zeta+1}$.  $Sz(T(1/2, \mathcal{S}_{\omega^\zeta}))=\omega^{\omega^\zeta\omega}=\omega^{\omega^{\zeta+1}}$.  

Case $4$: $\xi=\omega^\zeta$, $\zeta$ a limit ordinal. There is no Banach space with this Szlenk index by Theorem \ref{constant reduction}  

Case $5$: $\xi=\omega^{\alpha_1}n_1+\ldots + \omega^{\alpha_k}n_k$, where $\alpha_k>0$ and $k>1$ or $k=1$ and $n_1>1$.  Let $\zeta=\omega^{\alpha_1}n_1+\ldots + \omega^{\alpha_k}(n_k-1)$ and let $\eta=\omega^{\omega^{\alpha_k}}$.  Then $\omega^\zeta \eta=\omega^\xi$.  Moreover, for any $\beta< \eta$, $\beta^\omega\leqslant \eta$.  Take $\beta_n\uparrow \eta$ and note $\zeta \geqslant \sup_n \beta_n^\omega$, so $$Sz(T(\mathcal{S}_\zeta, (2^{-n}, \fff_{\beta_n})))=\iota(\mathcal{S}_\zeta) \sup_n \iota(\fff_{\beta_n})^\omega= \omega^\zeta\eta= \omega^\xi.$$

\end{proof}

\subsection{Mixed Tsirelson spaces as upper envelopes}

\begin{theorem} If $X$ is an infinite dimensional Banach space with shrinking FDD $E$, there exists a mixed Tsirelson space $T$ so that $Sz(X)=Sz(T)$ and a blocking $F$ of $E$ which satisfies subsequential $T$ upper block estimates in $X$.  \label{tsirelson envelope}
\end{theorem}

The proof is a modification of Theorem $5.5$ of \cite{Causey1}.  

\begin{proof}

Let $Sz(X)=\omega^\xi$.  

Step $1$: We claim that for any $\rho\in (0,1)$, we can find some $0=m_0<m_1<\ldots$ and regular families $(\mathcal{K}_n)_{n\geqslant 0}$ so that if $M=(m_n)_{n\geqslant 1}$ and if $F_n=[E_k]_{m_{n-1}<k\leqslant m_n}$, then for any $n\in \nn$ and any $(x_i)_{i=1}^k\in \Sigma(F, X)\cap \mathcal{H}^X_{\rho^{n-1}}$, $[m_n, \infty)\cap (m_{\max \supp_F(x_i)})_{i=1}^k\in \mathcal{K}_n[\mathcal{K}_0].$  Note that if $\mathcal{G}_n=M^{-1}(\mathcal{K}_n)$, this condition implies that if $(x_i)_{i=1}^k\in \Sigma(F,X)\cap \mathcal{H}^X_{\rho^{n-1}}$, then $[n, \infty)\cap (\max \supp_F(x_i))_{i=1}^k \in \mathcal{G}_n[\mathcal{G}_0]$.  

We will choose these families according to the following cases:

\noindent Case $1$: In the case that $\xi=1$, $\mathcal{K}_0=\mathcal{S}_0$ and for $n>0$, $\mathcal{K}_n=[\fff_s]^n$.  In this case, for any $\theta\in (0,1)$, $T=T(\mathcal{G}_0, (\theta^n, \mathcal{G}_n))=T(\theta, \fff_s)$ has $Sz(T)=\omega$.  

\noindent Case $2$: $\xi=\omega^{\zeta+1}=\omega^\zeta\omega$, we will choose $\mathcal{K}_0=\mathcal{S}_0$ and for $n>0$, $\mathcal{K}_n=[\mathcal{S}_{\omega^\zeta s}]^n$ for some $s\in \nn$.  Then in this case, for any $\theta\in (0,1)$, $T=T(\mathcal{G}_0, (\theta^n, \mathcal{G}_n))=T(\theta, \mathcal{G}_1)=T(\theta, M^{-1}(\mathcal{S}_{\omega^\zeta s}))$ has $Sz(T)=\omega^{\omega^\zeta\omega}=\omega^\xi$.  

\noindent Case $3$: $\xi=\omega^{\alpha_1}n_1+\ldots + \omega^{\alpha_k}n_k$, where $\alpha_k>0$ and either $k>1$ or $k=1$ and $n_1>1$.  Then we will let $\zeta=\omega^{\alpha_1}n_1+\ldots+\omega^{\alpha_k}(n_k-1)$ and $\beta=\omega^{\omega^{\alpha_k}}$, so $\omega^\zeta\beta=\omega^\xi$.  We choose $\beta_n\uparrow \beta$ and have $\mathcal{K}_0=\mathcal{S}_\zeta$ and for $n>0$, $\mathcal{K}_n=\fff_{\beta_n}$.  Then for any $\theta\in (0,1)$, $T=T(\mathcal{G}_0, (\theta^n, \mathcal{G}_n))$ is such that $Sz(T)=\omega^\xi$, since $\iota(\mathcal{G}_0)=\omega^\zeta\geqslant \beta=\sup_{n\in \nn}\iota(\mathcal{G}_n)^\omega= \sup_{n\in \nn} \beta_n^\omega$.  

This is a complete lists of possible cases, so in each case, $Sz(T)=Sz(X)$.  

Step $2$: We prove that with these choices, if $\theta\in (\rho, 1)$, $F$ satisfies subsequential $T$ upper block estimates in $X$.  We first complete step $2$ and then return to step $1$.  Let $(x_i)$ be a normalized block sequence with respect to $F$.  Let $l_i=\max \supp_F(x_i)$.  Choose $a=(a_i)\in c_{00}$ and let $x=\sum a_ix_i$.  Choose $x^*\in S_{X^*}$ so that $x^*(x)=\|x\|$.  For each $n\geqslant 1$, let $$A_n=(i\in \supp(a): i<n, \rho^{n-1} \leqslant |x^*(x_i)| < \rho^{n-2}),$$ $$B_n^+=(i\in \supp(a): i\geqslant n, \rho^{n-1} \leqslant x^*(x_i)<\rho^{n-2}), $$ $$B_n^-=(i\in \supp(a):i\geqslant n, \rho^{n-1}\leqslant -x^*(x_i)< \rho^{n-2}).$$  Since $\rho^{n-1}\leqslant x^*(x_i)$ for each $i\in B_n^+$, $(x_i)_{i\in B_n^+}\in \Sigma(F, X)\cap \mathcal{H}^X_{\rho^{n-1}}$.  Since $n\leqslant B_n^+$ and $l_i\geqslant i$, $$(l_i)_{i\in B_n^+}=[n, \infty)\cap (l_i)_{i\in B_n^+}\in \mathcal{G}_n[\mathcal{G}_0].$$  This means $$\theta^n\sum_{i\in B_n^+}|a_i|\leqslant \Bigl\|\sum_{i\in B_n^+} a_i e_{l_i}\Bigr\|_T \leqslant \Bigl\|\sum a_ie_{l_i}\Bigr\|_T.$$  Similarly, $$\theta^n\sum_{i\in B_n^-}|a_i|\leqslant \Bigl\|\sum a_i e_{l_i}\Bigr\|_T.$$  Last, since $(e_i)$ is normalized and $1$-unconditional, and since $|A_n|<n$, $$\sum_{i\in A_n}|a_i|\leqslant (n-1)\|a\|_\infty \leqslant (n-1)\Bigl\|\sum a_ie_{l_i}\Bigr\|_T.$$  Since $\{A_n, B_n^+, B_n^-: n\in \nn\}$ partitions $(i\in \supp(a): x^*(x_i)\neq 0)$,  \begin{align*} \|x\| & = \sum_{n=1}^\infty \Bigl[\sum_{i\in A_n}a_ix^*(x_i)+\sum_{i\in B_n^\pm} a_ix^*(x_i)\Bigr] \\ & \leqslant \sum_{n=1}^\infty \Bigl(\sum_{i\in A_n} |a_i|\rho^{n-2} +\sum_{i\in B_n^{\pm}} |a_i|\rho^{n-2}\Bigr) \\ & \leqslant \Bigl\|\sum a_ie_{l_i}\Bigr\|_T \sum_{n=1}^\infty \Bigl((n-1)\rho^{n-2} + 2\rho^{n-2}\theta^{-n}\Bigr) \\ & = \Bigl(\frac{1}{(1-\rho)^2} + \frac{2\rho^{-1}}{\theta-\rho}\Bigr)\Bigl\|\sum a_i e_{l_i}\Bigr\|_T. \end{align*}

We last complete step $1$.  Let $2\delta_n=\rho^{n-1}+\rho^n$ and $2\mu_n=\rho^{n-1}-\rho^n$.   For each $n\in \nn$, let $$\mathcal{B}_n=\Sigma(E, X)\cap \mathcal{H}^X_{\delta_n}$$ and choose $\overline{\varepsilon}_n=(\overline{\varepsilon}_{i,n})_i$ non-increasing so that $\sum_i \varepsilon_{i,n}<\mu_n$.  Observe that $(\mathcal{B}_n)^{E,X}_{\overline{\varepsilon}_n} \subset \mathcal{H}^X_{\rho^n}.$  By Lemma \ref{Zsak}, this implies $$\iota(\tilde{\mathcal{B}_n})\leqslant I_{bl}((\mathcal{B}_n)^{E,X}_{\overline{\varepsilon}_n})\leqslant I_w(\mathcal{H}^X_{\rho^n})<Sz(X).$$  Here we have used that $E$ is shrinking, so bounded block sequences in $E$ are weakly null.  If $Sz(X)=\omega$, then choose $s\in \nn$ so that $I_w(\mathcal{H}^X_\rho)<s$ and note that $I_w(\mathcal{H}^X_{\rho^n})<s^n$ for all $n\in \nn$ by Theorem \ref{constant reduction}.  If $Sz(X)=\omega^{\omega^{\zeta+1}}=\omega^{\omega^\zeta+1}$, choose $s\in \nn$ so that $I_w(\mathcal{H}^X_\rho)<\omega^{\omega^\zeta s}$ and note that for each $n\in \nn$, $I_w(\mathcal{H}^X_{\rho^n})< \omega^{\omega^\zeta sn}$.  In the third case, with $\zeta, \beta$ as in case $3$ above, choose any $\beta_n<\beta$ so that $I_w(\mathcal{H}^X_{\rho^n})<\omega^\zeta \beta_n$.  Let $\mathcal{K}_n$ be as given in the cases above.  We no longer need to distinguish between the three cases.  

Let $M_0=\nn$ and, using Lemma \ref{Gasparis}, recursively choose $M_n\in \infin$ so that for each $n\in \nn$, $M_n\in [M_{n-1}]$ and either $$\tilde{\mathcal{B}_n}\cap [M_n]^{<\omega} \subset \mathcal{K}_n[\mathcal{K}_0] \text{\ \ or\ \ }\mathcal{K}_n[\mathcal{K}_0]\cap [M_n]^{<\omega} \subset \tilde{\mathcal{B}_n}.$$  But in each case, $$\iota(\mathcal{K}_n[\mathcal{K}_0] \cap [M_n]^{<\omega})= \iota(\mathcal{K}_0) \iota(\mathcal{K}_n)>\iota(\tilde{\mathcal{B}_n}),$$ so the first containment always holds.  Choose $m_n\in M_n$ so that $m_1<m_2<\ldots$, set $M=(m_n)_{n\geqslant 1}$, and let $m_0=0$.  With $F_n=[E_k]_{m_{n-1}<k\leqslant m_n}$, to finish, we only need to show that for $n\in \nn$ and $(x_i)_{i=1}^k\in \Sigma(F, X)\cap \mathcal{H}_{\rho^{n-1}}^X$, $[m_n, \infty)\cap (m_{\max \supp_F(x_i)})_{i=1}^k \in \mathcal{K}_n[\mathcal{K}_0]$.  Again, let $l_i=\max\supp_F(x_i)$.  We can find a small perturbation $(y_i)_{i=1}^k\subset B_X$ of $(x_i)_{i=1}^k$ so that \begin{enumerate}[(i)]\item $\|y_i-x_i\|< \mu_n$, \item $\ran_F(y_i)=\ran_F(x_i)$, \item $m_{l_i}=\max \supp_E(y_i)$. \end{enumerate}  The first two items guarantee that $(y_i)_{i=1}^k\in \Sigma(E, X)\cap \mathcal{H}^X_{\rho^{n-1}-\mu_n}=\Sigma(E,X)\cap \mathcal{H}^X_{\delta_n}$.  The last two items guarantee that $$(m_{l_i})_{i=1}^k=(m_{\max \supp_F(y_i)})_{i=1}^k=(\max \supp_E(y_i))_{i=1}^k\in \tilde{\mathcal{B}}_n.$$ Combining these gives that $(m_{l_i})_{i=1}^k\in \tilde{\mathcal{B}}_n$.   But $[m_n, \infty)\cap (m_{\max \supp_F(y_i)})_{i=1}^k \in [M_n]^{<\omega}$, so that $$[m_n, \infty)\cap (m_{\max \supp_F(y_i)})_{i=1}^k \in \tilde{\mathcal{B}}_n\cap [M_n]^{<\omega}\subset \mathcal{K}_n[\mathcal{K}_0].$$

\end{proof}

\subsection{Universal spaces} 

If $\mathcal{C}$ is a class of Banach spaces, we say the Banach space $U$ is \emph{universal} for the class $\mathcal{C}$ if every member of $\mathcal{C}$ is isomorphic to a subspace of $U$.  We say $U$ is \emph{surjectively universal} for $\mathcal{C}$ if every member of $\mathcal{C}$ is isomorphic to a quotient of $U$.  

For each $0\leqslant \xi<\omega_1$, let $\mathcal{C}_\xi$ denote the class of separable Banach spaces with Szlenk index not exceeding $\omega^\xi$.  In \cite{DF}, it was shown that for each $\xi$, there exists a Banach space $Y_\xi$ having separable dual which is universal for $\mathcal{C}_\xi$.  The results there were obtained using descriptive set theory, which do not yield an estimate on $Sz(Y_\xi)$.  In \cite{FOSZ}, it was shown that for each $\xi<\omega_1$, $Y_\xi$ can be taken to be in $\mathcal{C}_{\zeta+1}$, where $\zeta=\min\{\eta\omega: \eta\omega\geqslant \zeta\}$.  In \cite{Causey1}, the following estimate was obtained.  

\begin{theorem} For every $1\leqslant \xi<\omega_1$, there exists a Banach space $Z_\xi\in \mathcal{C}_{\xi+1}$ which is both universal and surjectively universal for $\mathcal{C}_\xi$.  
\label{existence}
\end{theorem}

It was not stated in \cite{Causey1} that this space is surjectively universal for the class $\mathcal{C}_\xi$, but it is contained within the proof.  It was shown there that if $X$ is a separable Banach space with $Sz(X)\leqslant \omega^\xi$, then $X$ is isomorphic to a quotient of a Banach space $Y$ which embeds complementably in $Z_\xi$, so $X$ is isomorphic to a quotient of $Z_\xi$.  Our goal here is to prove that this result is optimal.

\begin{theorem} For any $\xi<\omega_1$, there does not exist a member of $\mathcal{C}_\xi$ which is universal or surjectively universal for $\mathcal{C}_\xi$.  \label{optimal}

\end{theorem}

In \cite{Causey2}, it was shown that if $\xi<\omega_1$ and if $\mathcal{CR}_\xi$ denotes the class of separable, reflexive Banach spaces $X$ with $Sz(X), Sz(X^*)\leqslant \omega^\xi$, then there exists a Banach space $Z\in \mathcal{CR}_{\xi+1}$ which is universal and surjectively universal for $\mathcal{CR}_\xi$.  In the proof of \ref{optimal}, we will prove that if $Z\in \mathcal{C}_\xi$, there exists $X\in \mathcal{C}_\xi$ which is not isomorphic to any quotient of any subspace of $Z$.  If $\xi>0$, this space will be a mixed Tsirelson space.  In the proof, we will have the freedom to choose the families used in the construction of $X$ so that $X$ is reflexive and $Sz(X^*)=\omega$, so that actually $X\in \mathcal{CR}_\xi$.  Therefore we will actually prove that if $Z$ is either universal or surjectively universal for $\mathcal{CR}_\xi$, $Z\notin \mathcal{C}_\xi$, and the result in \cite{Causey2} concerning the existence of a member of $\mathcal{CR}_{\xi+1}$ universal for $\mathcal{CR}_\xi$ is also optimal.   

Of course, the $\xi=0$ cases of Theorems \ref{existence} and \ref{optimal} are trivial, since $Sz(X)=1=\omega^0$ if and only if $X$ is finite dimensional.  We outline the idea for each of the other cases.  We note that for each $p>1$, $Sz(\ell_p)=\omega$.  Moreover, a separable Banach space $X$ has $Sz(X)=\omega$ if and only if for some $p>1$, every normalized, weakly null tree on $X$ indexed by $\widehat{\fin}$ has a branch which is dominated by the $\ell_p$ basis.  This means the $\ell_p$, $p>1$, spaces form a sort of upper envelope for $\mathcal{C}_1$.  But among the spaces $\ell_p$, $p>1$, no one of these spaces sits atop all the others.  To see how this can be generalized to $\mathcal{C}_\xi$, we use the following 

\begin{proposition} Let $Z$ be a Banach space having separable dual. \begin{enumerate}[(i)]\item If $X$ is isomorphic to a subspace of $Z$, then there exists $K>0 $ so that for any $\varepsilon\in (0,1)$, $I_w(\mathcal{H}^X_\varepsilon)\leqslant I_w(\mathcal{H}^Z_{\varepsilon/K})$. \item If $X$ is isomorphic to a quotient of $Z$, then there exists $K>0 $ so that for any $\varepsilon\in (0,1)$, $I_w(\mathcal{H}^X_\varepsilon)\leqslant I_w(\mathcal{H}^Z_{\varepsilon/K})$. \item If $X, Z$ are Banach spaces having separable dual so that $I_w(\mathcal{H}^X_{2^{-n}})>I_w(\mathcal{H}^Z_{3^{-n}})$ for all $n\in \nn$, then $X$ is not isomorphic to any subspace of any quotient of $Z$. \end{enumerate}

\end{proposition}

\begin{proof}(i) Let $T:X\to Z$ be an isomorphic embedding and fix $a,b>0$ so that $$a^{-1}\|x\|\leqslant \|Tx\|\leqslant b\|x\|$$ for all $x\in X$.  Let $K=ab$.  If $\xi<I_w(\mathcal{H}^X_\varepsilon)$ and if $(x_E)_{E\in \widehat{\fff}_\xi}\subset B_X$ is a weakly null tree with branches in $\mathcal{H}^X_\varepsilon$, then one easily checks that $(b^{-1}Tx_E)_{E\in \widehat{\fff}_\xi}\subset B_Z$ is a weakly null tree with branches in $\mathcal{H}^Z_{\varepsilon/K}$.  

(ii) First we assume $X$ is isometrically a quotient of $Z$, and then apply part (i).  Let $T:Z\to X$ be a quotient map.  Then if $\xi<I_w(\mathcal{H}^X_\varepsilon)$, fix $(x_E)_{E\in \widehat{\fff}_\xi}\subset B_X$ weakly null with branches in $\mathcal{H}^X_\varepsilon$.  Then by applying Proposition \ref{perturb} and Lemma \ref{pruning lemma}, for any $0<\delta<\varepsilon$, we can find a pruned subtree $(y_E)_{E\in \widehat{\fff}_\xi}$ of $(x_E)_{E\in \widehat{\fff}_\xi}$ and a weakly null tree $(z_E)_{E\in \widehat{\fff}_\xi}\subset 3B_Z$ so that $\|Tz_E-y_E\|<\varepsilon/2$.  This implies that $(3^{-1}Tz_E)_{E\in \widehat{\fff}_\xi}\subset B_X$ has branches lying in $\mathcal{H}^X_{\varepsilon/6}$. Since $T$ is norm $1$, the weakly null tree $(3^{-1}z_E)_{E\in \widehat{\fff}_\xi}\subset B_Z$ has branches lying in $\mathcal{H}^Z_{\varepsilon/6}$, and $I_w(\mathcal{H}^Z_{\varepsilon/6})>\xi$.  

(iii) If $X$ is isomorphic to a subspace of a quotient of $Z$, then there exists $K$ so that for each $\varepsilon\in (0,1)$, $I_w(\mathcal{H}^X_\varepsilon) \leqslant I_w(\mathcal{H}^Z_{\varepsilon/K})$.  If we choose $n$ so large that $2^{-n}K>3^{-n}$, then $$I_w(\mathcal{H}^Z_{3^{-n}})< I_w(\mathcal{H}^X_{2^{-n}}) \leqslant I_w(\mathcal{H}^Z_{2^{-n}/K})\leqslant I_w(\mathcal{H}^Z_{3^{-n}}),$$ a contradiction.

\end{proof}

\begin{proof} [Proof of theorem \ref{optimal}] 

Case $1$: $\xi=1$.  Suppose $Z\in \mathcal{C}_1$.  Then $I_w(\mathcal{H}^Z_{3^{-1}})<k$ for some $k<\omega$.  Then by Theorem \ref{constant reduction}, for each $n\in \nn$, $I_w(\mathcal{H}^Z_{3^{-n}})< k^n$.   But the Tsirelson space $T=T( 2^{-1}, 
\fff_k)$ has $k^n< I_w(\mathcal{H}^T_{2^{-n}})<\omega$ for each $n\in \nn$.  This means $T\in \mathcal{C}_1$ cannot be isomorphic to a subspace of a quotient of $Z$.  

Case $2$: $\xi=\omega^{\gamma+1}$. Suppose $Z\in \mathcal{C}_{\omega^{\gamma+1}}$.  Then $I_w(\mathcal{H}^Z_{3^{-1}})= \xi < \omega^{\omega^{\gamma+1}}= \omega^{\omega^\gamma \omega}$.  This means there exists $k\in \nn$ so that $I_w(\mathcal{H}^Z_{3^{-1}})< \omega^{\omega^\gamma k}$.  By Theorem \ref{constant reduction}, for each $n\in \nn$, $I_w(\mathcal{H}^Z_{3^{-n}})< \omega^{\omega^\gamma kn}$.  But for any $n\in \nn$, the Tsirelson space $T=T( 2^{-1},\mathcal{S}_{\omega^\gamma k })$ has $\omega^{\omega^\gamma kn}< I_w(\mathcal{H}^T_{2^{-n}})$.  This means $T\in \mathcal{C}_{\omega^{\gamma+1}}$ cannot be isomorphic to a subspace of a quotient of $Z$.  

Case $3$: $\xi=\omega^\gamma$, $\gamma$ a limit ordinal.  By Theorem \ref{existence theorem}, $\sup_{X\in \mathcal{C}_\xi} Sz(X)=\omega^\xi$.  Therefore if $Z$ is either universal or surjectively universal for $\mathcal{C}_\xi$, then $Sz(Z)\geqslant \omega^\xi$.  But again by Theorem \ref{existence theorem}, there is no Banach space with this Szlenk index, so $Sz(Z)>\omega^\xi$.  

Case $4$: $\xi=\omega^{\alpha_1}n_1+\ldots + \omega^{\alpha_k}n_k$, where $\alpha_k>0$ and either $k>1$ or $k=1$ and $n_1>1$.   In this case, let $\zeta=\omega^{\alpha_1}n_1+\ldots + \omega^{\alpha_k}(n_k-1)$ and $\eta=\omega^{\alpha_k}$.  Fix $Z\in \mathcal{C}_\xi$. Then there exists a sequence $(\beta_n)\subset [0, \omega^\eta)$ so that $I_w(\mathcal{H}^Z_{3^{-n}})< \omega^\zeta \beta_n$.  But for each $n\in \nn$, the mixed Tsirelson space $T=T(\mathcal{S}_\zeta, (2^{-n}, \fff_{\beta_n}))$ satisfies $\omega^\zeta\beta_n< I_w(\mathcal{H}^T_{2^{-n}})$.  Then $T\in \mathcal{C}_\xi$ cannot be isomorphic to a subspace of a quotient of $Z$.

\end{proof}

\subsection{Injective tensor products}

If $X, Y$ are Banach spaces, we can consider the tensor product $X\otimes Y$ as a collection of finite rank operators mapping $Y^*$ into $X$.  That is, given $u=\sum_{i=1}^n x_i\otimes y_i$, $u(y^*)=\sum_{i=1}^n y^*(y_i)x_i$.  Of course, the adjoint $u^*$ of $u$ maps $X^*$ into the image of $Y$ in $Y^{**}$ under the canonical embedding, so we can equally well consider $u$ as a map from $X^*$ into $Y$.  We can endow $X\otimes Y$ with the operator norm and let $X\hat{\otimes}_\varepsilon Y$ denote the completion of $X\otimes Y$ with respect to the operator norm.   Of course, $X\hat{\otimes}_\varepsilon Y$ is contained within the space of compact operators from $Y^*$ to $X$.  

If $X$ has FDD $E$, then for any compact $u:Y^*\to X$, $P^E_{[1,n]}u\to u$ with respect to the operator norm.  This implies that if $Y$ also has FDD $F$, $P^E_{[1,n]}u(P^F_{[1,n]})^*\to u$ with respect to the norm topology. If $E, F$ are shrinking FDDs for $X,Y$, respectively, then $$H_n=\text{span}\{E_k\otimes F_j: \max \{k,j\}=n\}$$ defines a shrinking FDD for the injective tensor product $X\hat{\otimes}_\varepsilon Y$ \cite{Causey1}. Showing that this forms an FDD is straightforward, while showing that this FDD is shrinking involves a characterization of weak nullity in injective tensor products given in \cite{L}. For $u\in X\hat{\otimes}_\varepsilon Y$, the projection $P^H_{[1,n]}u$ is given by $P^E_{[1,n]}uP^{F^*}_{[1,n]}$.  We think about such $u$ as an infinite matrix the $j,k$ entry of which is a member of $E_j\otimes F_k$.  In this case, the projections $P^H_{[1,n]}$ are the $n\times n$ leading principal minors of this infinite matrix.  Then a block sequence $(u_n)$ with respect to $H$ can be considered as a sequence of square matrices so that there exist $0=k_0<k_1<\ldots$ so that $u_n$ is equal to its $k_n\times k_n$ leading principal minor, while its $k_{n-1}\times k_{n-1}$ leading principal minor is zero.  In this case, we can write $u_n=r_n+c_n$ so that $(r_n)$ is a sequence of successive rows and $(c_n)$ is a sequence of successive columns.   This simple decomposition leads to the following 

\begin{proposition}\cite{Causey1} Suppose $T$ is a Banach space with normalized, $1$-unconditional basis $(e_n)$.  Let $X, Y$ be Banach spaces with shrinking, bimonotone FDDs $E, F$ so that $E$ (resp. $F$) satisfies subsequential $C$-$T$ upper block estimates in $X$ (resp. $Y$).  Then the FDD $H$ for $X\hat{\otimes}_\varepsilon Y$ satisfies $2C$-$T$ upper block estimates in $X\hat{\otimes}_\varepsilon Y$.  

\end{proposition}

\begin{theorem} For any separable, non-zero Banach spaces $X, Y$, $$Sz(X\hat{\otimes}_\varepsilon Y)= \max\{Sz(X), Sz(Y)\}.$$ \label{tensor theorem} \end{theorem}

For this, we will need the following simple fact.   

\begin{proposition} If $(e_n)$ is a shrinking, normalized, $1$-unconditional basis for the space $T$, and if $F$ is a shrinking FDD for the Banach space $Z$ such that $F$ satisfies subsequential $C$-$T$ upper block estimates in $Z$, then $Sz(Z)\leqslant Sz(T)$.  \end{proposition}

If $\xi<Sz(Z)$, then we can find for some $\varepsilon>0$ a block tree $(z_E)_{E\in \widehat{\fff}_\xi}$ with branches in $\mathcal{H}^Z_\varepsilon$.  Then if $t_E= e_{\max \supp_F (z_E)}$, $(t_E)_{E\in \widehat{\fff_\xi}}\subset B_T$ is a weakly null tree in $T$ with branches lying in $\mathcal{H}^T_{\varepsilon/C}$.  This means $\xi<Sz(T)$.

\begin{proof}[Proof of Theorem \ref{tensor theorem}] If either $X^*$ or $Y^*$ is non-separable, the result is clear. If either space is finite dimensional, the result is immediate from Theorem \ref{minkowski}, since in this case the tensor product is isomorphic to a finite direct sum where each summand is either $X$ or $Y$.  Assume $Sz(X), Sz(Y)<\omega_1$.  

If $X$ is a closed subspace of $X_0$ and $Y$ is a closed subspace of $Y_0$, then $X\hat{\otimes}_\varepsilon Y$ is isomorphic to a subspace of $X_0\hat{\otimes}_\varepsilon Y_0$.  By a result of Schlumprecht \cite{S}, we can embed $X, Y$ into Banach spaces $X_0, Y_0$ with shrinking, bimonotone bases so that $Sz(X)=Sz(X_0)$ and $Sz(Y)=Sz(Y_0)$.  Thus it suffices to assume that $X, Y$ themselves have shrinking, bimonotone bases.  Then by Theorem \ref{tsirelson envelope}, we can find Banach spaces $T_X$, $T_Y$ with normalized, $1$-unconditional bases $(e^X_n), (e^Y_n)$ so that $Sz(T_X)=Sz(X)$ and $Sz(T_Y)=Sz(Y)$ and shrinking, bimonotone FDDs $E, F$ for $X, Y$, respectively so that $E$ satisfies subsequential $T_X$ upper block estimates in $X$ and $F$ satisfies subsequential $T_Y$ upper block estimates in $Y$.  Then if $e_n=e^X_n+e^Y_n\in T_X\oplus_\infty T_Y$, $E, F$ satisfy subsequential $[e_n]$ upper block estimates in $X, Y$, respectively.  Therefore the FDD $H$ is a shrinking FDD for $X\hat{\otimes}_\varepsilon Y$ satisfying subsequential $[e_n]$ upper block estimates in $X\hat{\otimes}_\varepsilon Y$.  We deduce that $$Sz(X\hat{\otimes}_\varepsilon Y)\leqslant Sz([e_n]) \leqslant Sz(T_X\oplus_\infty T_Y) =\max\{Sz(T_X), Sz(T_Y)\} \leqslant \max\{Sz(X), Sz(Y)\}.$$  

\noindent Since $X, Y$ both embed into $X\hat{\otimes}_\varepsilon Y$, the reverse inequality is clear.

\end{proof}

\begin{remark} It is unnecessary to take the direct sum $T_X\oplus_\infty T_Y$ in the previous proof.  It is easy to see how to modify the proof of Theorem \ref{tsirelson envelope} to find one mixed Tsirelson space which can play the roles of both $T_X$ and $T_Y$ simultaneously.

\end{remark}


\begin{thebibliography}{HD}

\normalsize
\baselineskip=17pt

\bibitem{AA}  D. Alspach,  S. Argyros, \emph{Complexity of weakly null sequences}, Diss. Math., 321 (1992), 1-44.

\bibitem{AJO} D. Alspach, R. Judd, E. Odell, \emph{The Szlenk index and local $\ell_1$-indices}, Positivity, 9 (2005), no. 1,1-44. 

\bibitem{AT} S. Argyros, S. Todorcevic, \emph{Ramsey theory in Analysis}, Birkh\"{a}user Verlag, Basel, 2005.  

\bibitem{Bo} J. Bourgain. \emph{On convergent sequences of continuous functions}, Bull. Soc. Math. Bel., 32 (1980), 235-249.

\bibitem{Brooker} P.H. Brooker, \emph{Asplund operators and the Szlenk index}, Operator Theory 68 (2012), 405-442.

\bibitem{BL} P.H. Brooker, G. Lancien, \emph{Three space properties and asymptotic structure of Banach spaces}, J. Math. Anal. Appl. 398(2), 2013, 867-871.

\bibitem{Causey1} R.M. Causey, \emph{Estimation of the Szlenk index of Banach Spaces via Schreier spaces}, Studia Math. 216 (2013), 149-178. 

\bibitem{Causey2} R.M. Causey, \emph{Estimation of the Szlenk index of reflexive Banach spaces using generalized Baernstein spaces}, to appear in Fundamenta Mathematicae. 

\bibitem{DF} P. Dodos, V. Ferenczi, \emph{Some strongly bounded classes of Banach spaces}, Fund. Math. 193 (2007), no. 2, 171-179.

\bibitem{FOSZ} D. Freeman, E. Odell, Th. Schlumprecht, A. Zs\'{a}k, \emph{Banach spaces of bounded Szlenk index II}, Fund. Math. 205 (2009) 161-177.

\bibitem{G} I. Gasparis, \emph{A dichotomy theorem for subsets of the power subsets of the power set of the natural numbers}, Proceedings of the American Mathematical Society, 129 (2001) 759-764. 

\bibitem{James} R. C. James, \emph{Uniformly non-square Banach spaces}, Ann. of Math. (2), 80:542-550, 1964. 

\bibitem{JO} R. Judd, E. Odell, \emph{Concerning Bourgain's $\ell_1$ index of a Banach space}, Israel J. Math 108(1998), 145-171. 

\bibitem{La1} G. Lancien, \emph{On the Szlenk index and the weak$^*$ dentability index}, Quarterly J. Math. Oxford, 47, 1996, 59-71.  

\bibitem{LT} D. H. Leung, W.K, Tang, \emph{The Bourgain $\ell^1$-index of mixed Tsirelson space}, J. Funct. Anal. 199 (2003), no. 2, 301-331.

\bibitem{L} D. R. Lewis, \emph{Conditional weak compactness in certain inductive tensor products}. Math. Ann. 201 (1973) 201-209.  


\bibitem{OS} E. Odell and Th. Schlumprecht, \emph{Embedding into Banach spaces with finite dimensional decompositions}, Rev. R. Acad. Cienc. Exactas Fis. Nat. Ser. A Mat. vol 100 (1-2)(2006), 295-323.
  

\bibitem{OSZ} E. Odell, Th. Schlumprecht, A. Zs\'{a}k.  \emph{Banach spaces of bounded Szlenk index}, Studia Math. 183 (2007), no. 1, 63-97. 

\bibitem{PR}  P. Pudl\'{a}k and V. R\"{o}dl, \emph{Partition theorems for systems of finite subsets of
integers}, Discrete Math. 39 (1982), no. 1, 67-73.

\bibitem{Rosenthal} H. Rosenthal.  \emph{A characterization of Banach spaces containing $\ell^1$} Proc. Nat. Acad. Sci. 71 (1974), 2411-2413. 

\bibitem{S} Th. Schlumprecht. \emph{On Zippin's embedding theorem}, submitted.  

\bibitem{Sz} W. Szlenk, \emph{The non existence of a separable reflexive Banach space universal for all separable reflexive Banach spaces}, Studia Math. 30 (1968), 53-61. 


\end{thebibliography}
\end{document}